\documentclass[a4paper,10pt]{article}
\pdfoutput=1
\usepackage[utf8]{inputenc}
\usepackage[all]{xy}
\usepackage{amsmath,latexsym,amssymb,amsthm,euscript,faktor,animate}
\usepackage[nottoc]{tocbibind}
\usepackage[final]{graphicx}

\usepackage{float}
\usepackage{subcaption,afterpage}
\usepackage[english]{babel}
\usepackage{pgfplots}
\usepgfplotslibrary{polar}
\usepackage{tikz,tikz-cd,scalerel}
\usepackage{etoolbox}
\usetikzlibrary{matrix, arrows}
\newcommand {\R}{\mathbb R}

\newcommand {\Z}{\mathbb Z}
\newcommand {\N}{\mathbb N}
\newcommand {\CN}{\mathbb C}
\newcommand {\CP}{{\mathbb C}{\bf P}}

\newcommand\scalescale[2]{\hstretch{#1}{\vstretch{#1}{#2}}}
\newcommand\Preccurlyeq{\scalescale{1.5}{\preccurlyeq}}
\theoremstyle{definition}
\newtheorem{definition}{Definition}
\theoremstyle{plain}
\newtheorem{theorem}{Theorem}
\newtheorem{lemma}{Lemma}

\newtheorem{example}{Example}
\newtheorem{problem}{Question}
\newtheorem{proposition}{Proposition}
\makeatletter
\tikzset{
  edge node/.code={%
      \expandafter\def\expandafter\tikz@tonodes\expandafter{\tikz@tonodes #1}}}
\makeatother
\tikzset{
  preccurlyeq/.style={
    draw=none,
    edge node={node [sloped, allow upside down, auto=false]{$\Preccurlyeq$}}},
  Preccurlyeq/.style={
    draw=none,
    every to/.append style={
      edge node={node [sloped, allow upside down, auto=false]{$\Preccurlyeq$}}}
  }}
\title{Continuity argument revisited: geometry of root clustering via symmetric products}
\author{Grey Violet\footnote{This work were partially done during author's Marie Curie Incoming Fellowship at the University of Konstanz.}\\Zukunftskolleg and Department of Mathematics and Statistics.\\ University of Konstanz. Konstanz, Germany.\\ navernoeinikogda@gmail.com}
\date{\today}

\begin{document}
\maketitle

\begin{abstract}
We study the spaces of polynomials stratified into the sets of polynomial with fixed number of roots inside certain semialgebraic region $\Omega$, on its border, and at the complement to its closure. Presented approach is a generalisation, unification and development of several classical approaches to stability problems in control theory: root clustering ($D$-stability) developed by R.E. Kalman, B.R. Barmish, S. Gutman et al., $D$-decomposition(Yu.I. Neimark, B.T. Polyak, E.N. Gryazina) and universal parameter space method(A. Fam, J. Meditch, J.Ackermann).
    
Our approach is based on the interpretation of correspondence between roots and coefficients of a polynomial as a symmetric product morphism.

We describe the topology of strata up to homotopy equivalence and, for many important cases, up to homeomorphism. Adjacencies between strata are also described. Moreover, we provide an explanation for the special position of classical stability problems: Hurwitz stability, Schur stability, hyperbolicity.
\end{abstract}\newpage
\tableofcontents

\section{Introduction}

The problem of root clustering \-- study of the sets of polynomials or matrices with fixed distribution of roots (eigenvalues) relative to some region $\Omega$ is one of the classical problems in control theory. 

Namely, asymptotic stability for linear continuous-time systems introduced in 1868 by J.C. Maxwell in the very first paper on control theory \cite{Max1868} led to the study of the distribution of roots relative to the Hurwitz stability region \--- open left half-plane $\Omega=\{Re\, z<0\},$  asymptotic stability for the linear discrete-time systems have led to the distributions of roots relative to the open unit disk(Schur stability) $\Omega=\{|z|<1\},$ \cite{Sch1917}, real-rootedness (hyperbolicity) property of a polynomial, localisation of roots relative to some finite number of points, are among the most classical situations.

General root clustering problems in control have been systematically studied since the appearance of a seminal paper by R.E Kalman \cite{Kal1969} from 1969. That moment could be seen as a start of development of {\it root clustering theory} as a special topic in control. Since then many different types of stability regions have been examined in many different contexts: sufficient condition for Hurwitz and Schur superstabilizability  \cite{PS2002} and other sectoral stability regions, aperiodicity( in a form $\Omega=\{z\in\R, z<0\}$) \cite{NP1994}, unions of disjoint disks \cite{AY1997}, plane curves \cite{Maz1999}, Cassini ovals \cite{MD2014},  $LMI$-regions(i.e. rigidly convex ones \cite{HV2007}) \cite{CG1996,CGA2002}, Ellipsoidal Matrix Inequality Regions ($EMI$-regions) \cite{Peauc2000,PABB2000}, Extended Ellipsoidal Matrix Inequality Regions ($EEMI$-regions) \cite{BM2003}, Boolean combinations of different regions\cite{BHPM2004,BBM2005}. Among these multiple contributions should be especially noted a general theory of root clustering that has been developed by E. Jury and S. Gutman since the beginning of 1980's (see \cite{GJ1981} for formulation of the foundations and book \cite{Gut1990}). B.R. Barmish \cite{Bar1993} also consider close theory of $D$-stability in his classical book on robust control. Due to a vast amount of literature on the subject we are leaving many references to the (very partial) historical review below.

Examples of different root clustering region could be seen on Fig. \ref{stabreg}.


Despite that no general geometric theory of that kind of problems is known up to date, there were several attempts of different geometric approaches to this type of problems. Among those one should note such geometric approaches to the Hurwitz stability as $D$-decomposition, with study have been started by A.A. Andronov and Yu.I. Neimark in 1940-s \cite{Nei1949}, with the  most recent developments being attributed to E.N. Gryazina, B.T. Polyak \cite{GP2006,GPT2008} in 2000's and the author in 2010's \cite{Vas2012a,Vas2012b,Vas2014}. That approach consists in study the decomposition of parameter space for a fixed affine polynomial family 

$$
a_0f_0(x)+\ldots +a_kf_k(x)
$$

into the regions with fixed number of Hurwitz stable \--- located inside the Hurwitz stability region. and Hurwitz unstable roots \--- located outside of the closure of the stability region. This decomposition is called {\it $D$-decomposition}. A number of algorithms (both analytic and algebraic \-- related to the Quantifier Elimination problems), complicated examples and topological complexity bounds have been obtained within the framework of that approach.

Another important geometric approach is universal parameter space theory, which development have been started since A. Fam and J. Meditch from 1978 \cite{FM1978}.
Here topology and algebraic geometry of the space of Hurwitz and  Schur stable polynomials \--- polynomials with all of their roots belonging to the root clustering region, for the case of indeterminate polynomial family has been considered. 

Our line of research consists in unification and generalisation of all 3 approaches considered before. Namely, that paper is the start of the development of a general geometric theory of root clusterisations for any semialgebraic region $\Omega.$ Here, we consider a refinement of the $D$-decomposition of the parameter space for an indeterminate polynomial, that we a call a {\it $D$-stratification}. $D$-stratification consists in the decomposition of parameter space into the region with fixed number of roots inside the $\Omega,$ on the border of $\Omega$ and in the complement to the closure of $\Omega$ for an arbitrary semialgebraic $\Omega.$

Instead of concentrating on the study of different very special regions $\Omega$ or very special families of polynomials or providing some purely algorithmic root clustering criteria, we are trying to figure out some limited number of natural geometric constructions lying behind that kind of problems and sometimes informally used in classical papers as a {\it continuity argument}. That paper is a first part of a longer text, that will cover not not only topological properties of the sets in question, but also their singularity theory, convexity-like properties, metric properties and construction of stratified spaces parametrising stability problems.  Matrix problems are only slighlty mentioned there.

Main of these natural geometric construction is the symmetric product construction. Namely, our idea consists in transferring simple planar semialgebraic geometry and topology of the stability region $\Omega$ across the symmetric product functor, to the parameter space. Through the means of this transfer we become able to compute the topology of the $D$-strata. In order to show the computability of all constructions we prefer to use a semialgebraic category.

That paper is devoted mainly to the topological study of problems in question, as well as to the development of basic formalism covering main properties and classical intuitions behind this set of problems.


We prove (Theorem \ref{classicalpoly}) that the stratification of $\CP^1$ with respect to the $\Omega$ transferred via symmetrization induces a $D$-stratification on the space of coefficients of an indeterminate polynomial. Stratum of the $D$-stratification is a set of all polynomials with fixed number of stable, semi-stable and unstable roots. Using that construction and different results on topology of symmetric products we describe topology of $D$-strata and adjacency relations between them.

In order to show computability of all constructions and in order to make a later use of the developed formalism in the study of singularities we make all work in a semialgebraic category. 

We use here the language of abstract semialgebraic sets \-- semialgebraic spaces. It is essential because of Lemmata \ref{unmixing},\ref{unmixingconnected} where one need to consider a morphism from some semialgebraic subset of the product of complex projective spaces, which couldn't be extended to the variety as a whole. This is the main instrument of transfer of the planar geometry of the root space into the complicated geometry of the space of coefficients, we mentioned before.

It is important to state that we are able to capture not only the geometry of a fixed degree polynomial, but also a local geometry of {\it degree change}, capturing, therefore a classical intuition of births and deaths of roots at the infinity, which forms an important part of a classical intuition. That geometry is captured by the structure of {\it filtered space} on the space of indeterminate polynomials, which happens to agree with $D$-stratification. Interaction between $D$-stratification and filtered structure is examined in Proposition \ref{birthdeath1} and Theorem \ref{adjdigstab}. We also introduce an adjacency digraph functor allowing us to compute structure of adjacencies between $D$-strata in a natural way (Lemma \ref{funct}, Proposition \ref{spoweradj}).

Moreover, our construction of the root-coeffient correspondence as a symmetric product of stratified variety allows to obtain a natural duality between spaces of monic and non-monic polynomial, which happens to be also a duality between polynomials and matrices (Theorem \ref{matrixpoly}). In the matrix case conjugation action of the group $Gl(\CN,n)$ on matrices appears as  in certain sense dual to the action of symmetric group on the space of polynomial roots. Here we are not going into further consideration of the subject, just noting that this duality is produced by two types of degree-changing perturbations of a polynomial.

$$
\begin{array}{l}
\text{Non-monic polynomials:}\quad\ a_nz^n+\ldots +a_0\mapsto \epsilon z^{n+1}+a_nz^n+\ldots+a_0\\
\text{Matrices and monic polynomials:}\,\, a_nz^n+\ldots +a_0\mapsto z(a_nz^n+\ldots +a_0)+\epsilon.
\end{array}
$$
In order to make above-mentioned apparatus work we introduce category of filtered stratified real algebraic varieties and their semialgebraic subspaces. That formalism, including (symmetric) group actions on those structures is developed in the Section 3.

Formalisms of stratified filtered real algebraic varieties and formalism of stability theories and $D$-stratifications allows us to describe the topology of $D$-strata up to homotopy equivalence, where $D$-strata appearing to be described via products of  and, in many cases, up to homeomorphism. One example result, appearing in it's full formulation in Theorem \ref{hott} and Proposition \ref{compactconnected} is mentioned below:

\begin{figure}[H]
\subcaptionbox{Hurwitz stability theory: $\Omega=\{Re\,z<0\}.$}{\includegraphics[width = 2.1in]{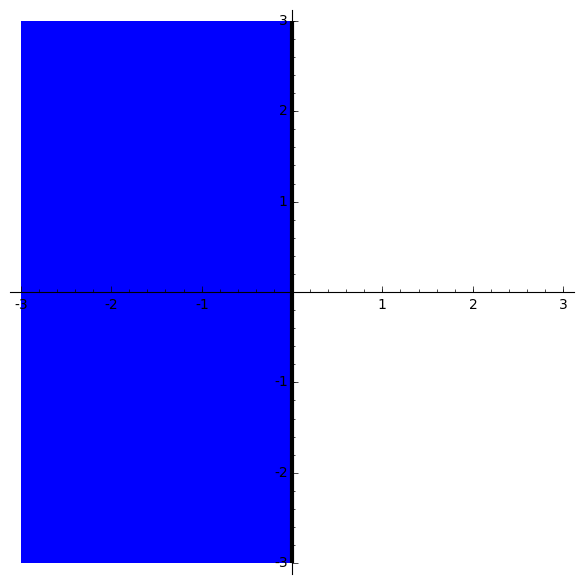}}\qquad
\subcaptionbox{Schur stability theory: $\Omega=\{|z|<1\}.$}{\includegraphics[width = 2.1in]{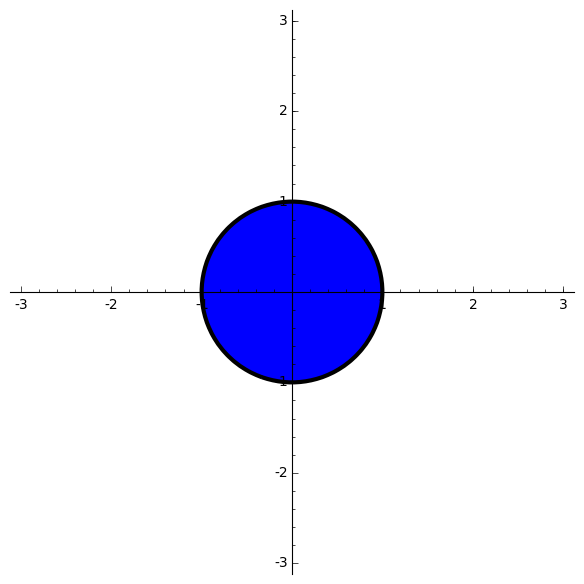}}\\   
\subcaptionbox{Sectoral stability. Sufficient condition for Hurwitz superstabilizability \protect{\cite{PS2002}}. $\Omega=\{Re\,z+|Im\,z|<0\}.$}{\includegraphics[width = 2.1in]{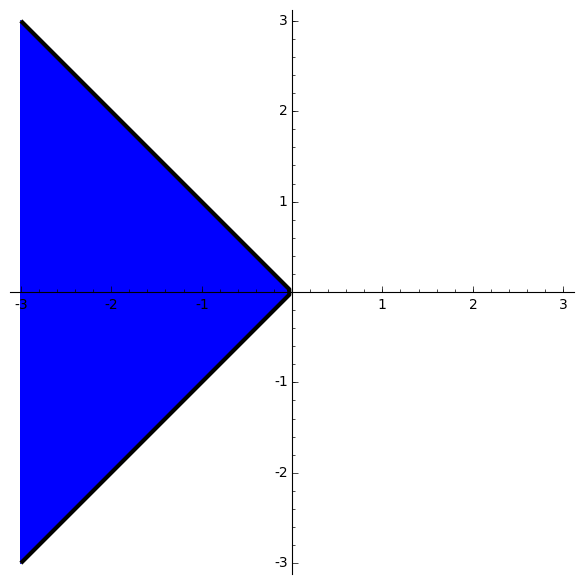}}\qquad
\subcaptionbox{Sufficient condition for Schur superstabilizability \protect{\cite{PS2002}}. $\Omega=\{|Re\,z|+|Im\,z|<1\}.$}{\includegraphics[width = 2.1in]{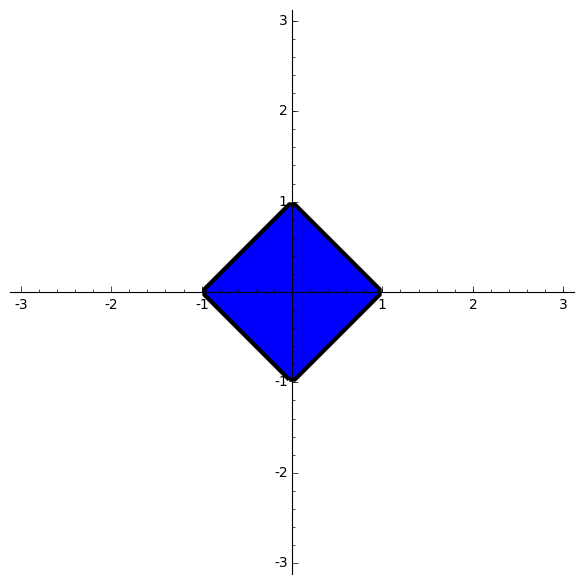}}\\
\subcaptionbox{Hyperbolicity: $\Omega=\{Im\, z<0\},$ $\Omega_{ss}=\R.$}{\includegraphics[width = 2.1in]{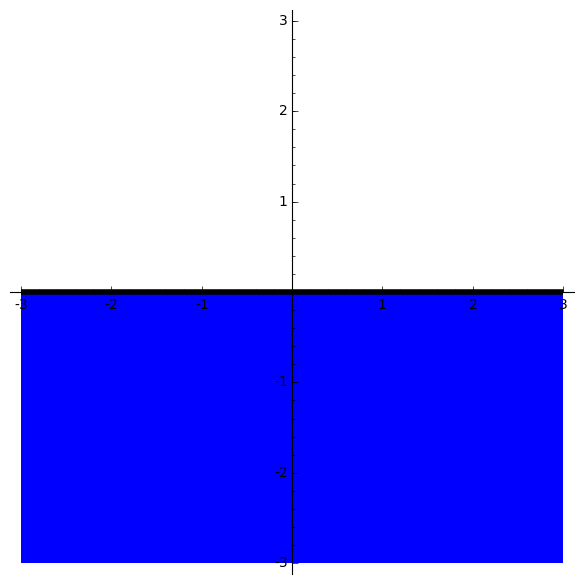}}\qquad
\subcaptionbox{Ride quality: $\Omega=\{0.6<|z|^2<1,\,-0.5<Im\,z<0.5,\,Re\,z<0\}$ \protect{\cite{GJ1981}}.}{\includegraphics[width = 2.1in]{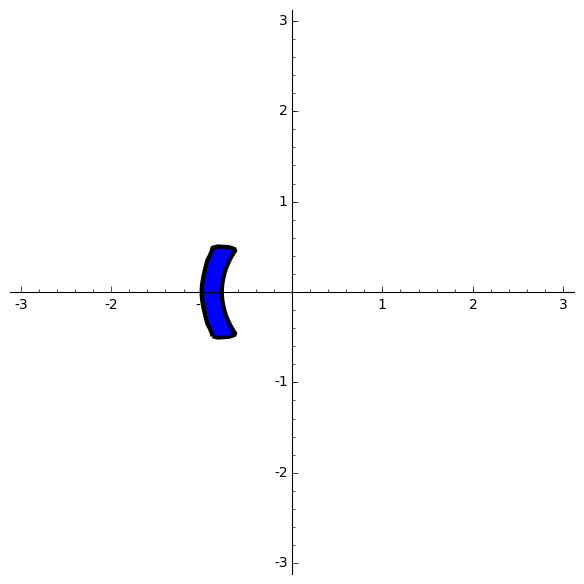}}
\phantomcaption
\end{figure}
\begin{figure}[H]
\ContinuedFloat
\subcaptionbox{Pole placement. $\Omega$ is finite.}	{\includegraphics[width = 2.1in,height = 2.1in]{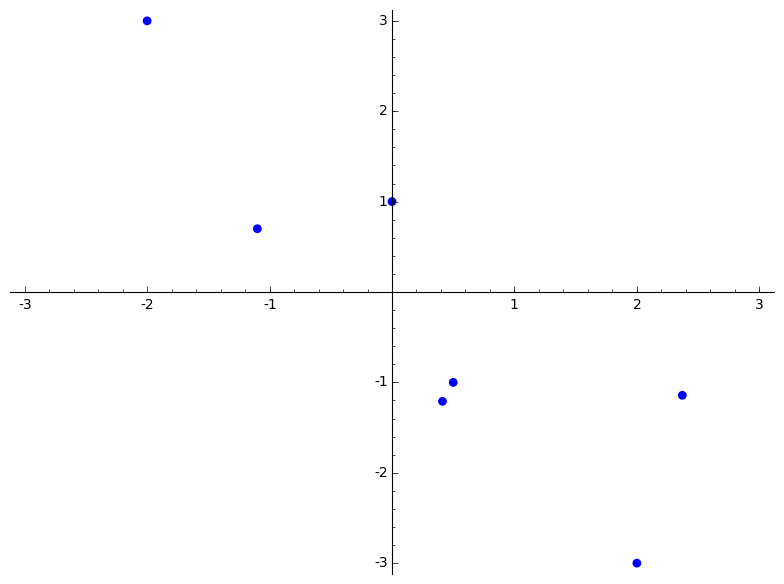}}\qquad
\subcaptionbox{Aperiodicity: $\Omega=\{z\in\R, z<0\}.$ \protect{\cite{NP1994}}}{\includegraphics[width = 2.1in,height = 2.1in]{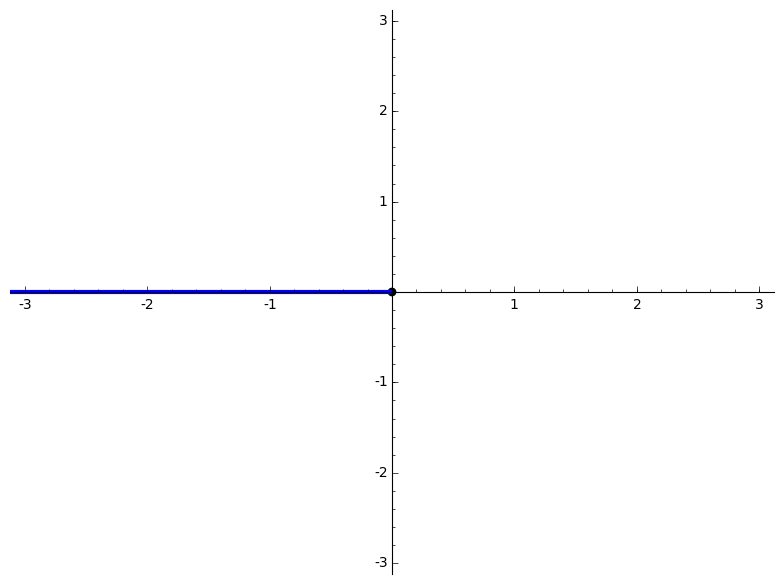}}\\
\subcaptionbox{Fenichel stability: $\Omega=\{Re\,z>0\vee Re\,z<0\}$ \protect{\cite{Fen1979}}.}{\includegraphics[width = 2.1in]{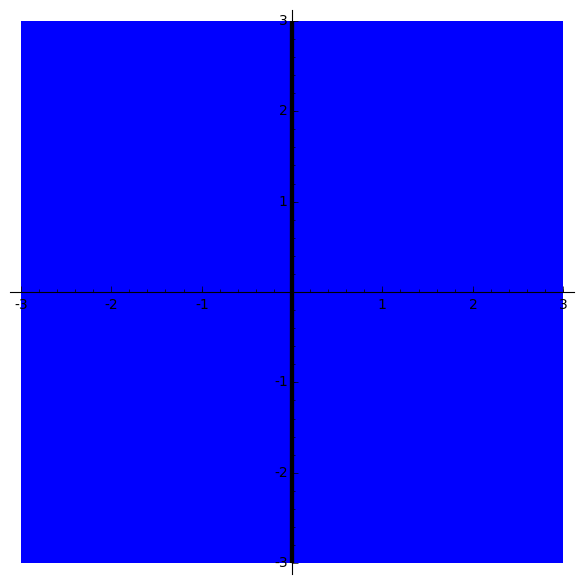}}\qquad
\subcaptionbox{Bounded frequency: $\Omega=\{-2<Im\,z<2\}$ \protect{\cite{GJ1981}}.}{\includegraphics[width = 2.1in]{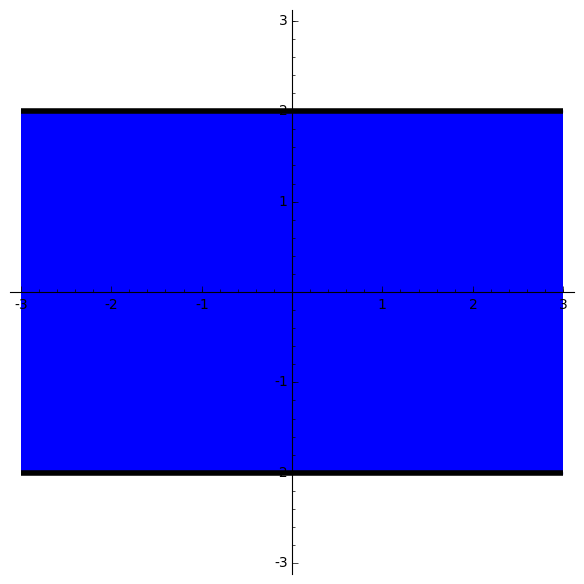}}
\caption{Different types of root clustering regions, connected with control problems. Stability region is blue, semi-stability region is black, unstable region is white.}
\label{stabreg}
\end{figure}

\begin{theorem}
Let stable region $\Omega$ be an open semialgebraic set. Then connected component of a	 $D$-stratum is homotopy equivalent to the direct product of skeleta of tori.

Here by $q$-skeleton of $n$-tori $T^n$ we mean the set $$\bigcup_{I\subseteq \{1,\ldots,n\}, |I|=q}\{(s_1,\ldots, s_n)\in T^n| \forall i\in I\,s_i=1\}.$$ 
Examples of skeleta could be seen on Fig. \ref{Bonesoftorus}.

Moreover, if $\Omega$ is convex and connected then each $D$-stratum is either contractible or homeomorphic to the product of euclidean space and disc bundle over circle.

Disc bundle is orientable if $D$-stratum contains polynomials with odd number of zeros on the border of $\Omega,$ and non-orientable if number of zeros on the border of $\Omega$ is even.
\end{theorem}

Main ideas of the proof contains a correspondence between stability problems for the finite $\Omega$ and general position hyperplane arrangements noted for the first time by B.W. Ong, in a context of computing symmetri products of bouquets of circles \cite{Ong2003}. 

One can also compute Betti numbers of a $D$-stratum, in particular, number of connected components of a $D$-stratum.

\begin{theorem}
Let stable region $\Omega$ be open in its closure.

Denote as $b_s$ number of connected components of $\Omega,$ as $b_{ss}$\--- number of connected components of the border of $\Omega,$ as $b_{un}$ number of connected components of the complement to a closure of $\Omega.$
 
 Then the space of polynomials with $k$ roots in $\Omega,$ $l$ roots in $\partial\Omega,$ and $m$ roots in $\CN\setminus\Omega$ has 
 $${b_{s}+k-1\choose k}{b_{ss}+l-1\choose l}{b_{un}+m-1\choose m}$$ connected components.
\end{theorem}
General formulations of the result are contained in Theorem \ref{betti} and Proposition \ref{connectcomp}.

Moreover, one can describe adjacencies between $D$-strata.

\begin{theorem}
Let stability set $\Omega$ be a semialgebraic set. Denote as $Str(\Omega)$ set $\{s=\Omega,ss=\overline{\Omega}\setminus\Omega,un=\CP^1\setminus\overline{\Omega}\}.$
 
Let $\tau=(t_s,t_{ss},t_{un})$ and $\eta=(q_s,q_{ss},q_{un}),$ $\sum_{i\in Str(\Omega)}t_i=\sum_{i\in Str(\Omega)}^kq_i=m$ be $D$-strata. 
 Value of each component of an ordered triple equals to the number of roots situated in the corresponding region.
 Define $Adj(Str(\Omega))$ as $\{t=(t_1,t_2)\in Str(S)^2|\overline{t_1}\cap t_2\neq\varnothing\}$
%
 $\overline{\tau}\cap\eta\neq\varnothing$ iff there exists such a fa\-mi\-ly of na\-tu\-ral num\-bers $$\{\mu_t|t\in Adj(Str(\Omega))\}$$ that the fol\-lo\-wing sys\-tem of li\-near equa\-tions has a so\-lu\-ti\-on:
$$
\begin{array}{l}
\sum_{j\colon (j,i)\in E(Adj(Str(\Omega)))}\mu_{(j,i)}=e_i,\\
\sum_{j\colon (i,j)\in E(Adj(Str(\Omega)))}\mu_{(i,j)}=t_i,\qquad i\in\{1,2,3\}.
\end{array}
$$
\end{theorem}

General formulation and proof one may find in Theorem \ref{adjcrit}.

Invariance under both different types of degree-changing perturbations alongside with condition for a root clustering theory to behave naturally on the real-coefficient polynomials allows to provide an explanation of the special position of classical stabilities: Hurwitz stability $\Omega=\{Re\, z<0\}$, Schur stability $\Omega=\{|z|<1\}$, hyperbolicity $\Omega=\{Im\, z=0\}$.

That theorem could be formulated as follows:
\begin{theorem}[Standard stability theories]
 Let $\Omega$ be a non-empty open se\-mi\-al\-ge\-bra\-ic set on $\CP^1.$
 Con\-si\-der a stra\-ti\-fi\-ca\-tion of $\CP^1$ into sets $\Omega$, $\partial\Omega$, $\CP^1\setminus\overline{\Omega}.$
 
Suppose that:
 \begin{enumerate}
  \item $\partial\Omega$ is an irreducible real algebraic curve without isolated points.
  \item Inversion and complex conjugation are automorphisms of the stratified space. 
  \item $0$ and $\infty$ both lie on $\partial\Omega$ or they are in different strata.
 \end{enumerate}
 Then $\partial \Omega$ is one of the coordinate lines or a unit circle.
\end{theorem}
\begin{figure}[H]
\centering
\subcaptionbox{$T_1^6$ is the bouquet of $6$ circles}{ \begin{tikzpicture}[scale=0.8]
   \begin{polaraxis}[grid=none, axis lines=none]
     \addplot[mark=none,domain=0:360,samples=6500] { abs(cos(6*x/2))};
   \end{polaraxis}
 \end{tikzpicture}}\\
\subcaptionbox{Intersection poset for the coordinate subtori of $T_2^3.$ Each pair of $2$-dimensional tori intersects by different coordinate circle. All this circles are intersecting in a single point.}
{\begin{tikzpicture}
\node[inner sep=0pt] (lefttorus) at (0,6)
    {\includegraphics[width=.1\textwidth]{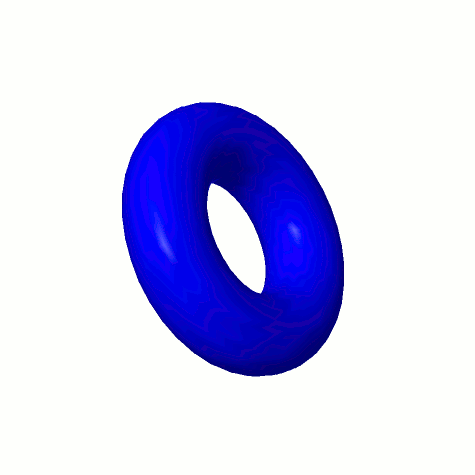}};
\node[inner sep=0pt] (centertorus) at (2,6)
    {\includegraphics[width=.1\textwidth]{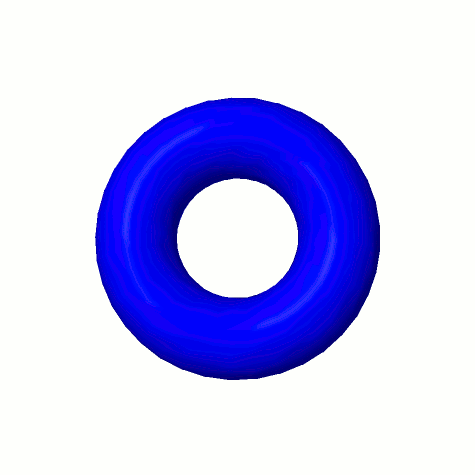}};
\node[inner sep=0pt] (righttorus) at (4,6)
    {\includegraphics[width=.1\textwidth]{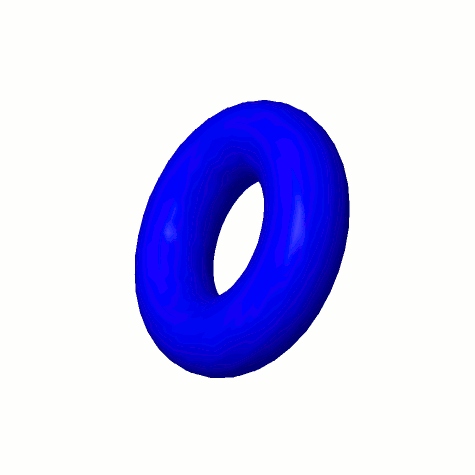}};
\node[inner sep=0pt] (leftcircle) at (1,3)
    {\includegraphics[width=.05\textwidth]{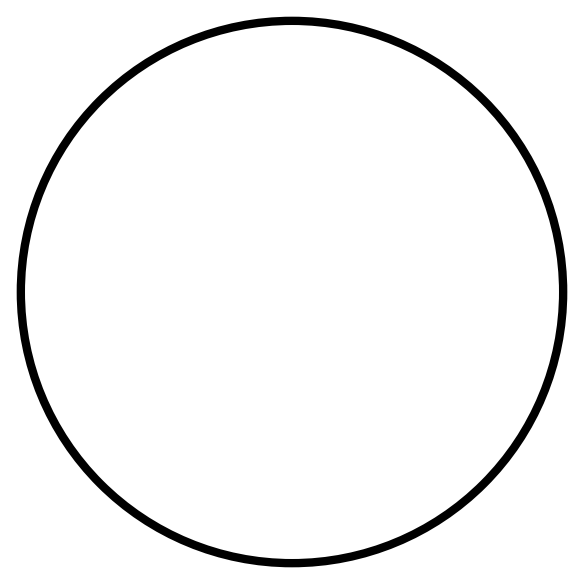}}; 
\node[inner sep=0pt] (rightcircle) at (3,3)
    {\includegraphics[width=.05\textwidth]{./figures/circle.png}}; 
\node[inner sep=0pt] (centercircle) at (2,2)
    {\includegraphics[width=.05\textwidth]{./figures/circle.png}};
\node[inner sep=0pt] (point) at (2,0)
     {{$\bullet$}};
\draw[-,thin] (lefttorus.south) -- (leftcircle.north);
\draw[-,thin] (centertorus.south) -- (leftcircle.north);
\draw[-,thin] (centertorus.south) -- (rightcircle.north);
\draw[-,thin] (righttorus.south) -- (rightcircle.north);
\draw[-,thin] (lefttorus.south) -- (centercircle.north);
\draw[-,thin] (righttorus.south) -- (centercircle.north);
\draw[-,thin] (leftcircle.south) -- (point.north);
\draw[-,thin] (rightcircle.south) -- (point.north);
\draw[-,thin] (centercircle.south) -- (point.north);
\end{tikzpicture}}\\
\caption{ Skeleta of tori.}
\label{Bonesoftorus}
\end{figure}

This is the Theorem \ref{standard}.

Some results about the structure of reducible borders of borders without condition on $0$ and $\infty$ are also proved.
Those results lead to the following question.

\begin{problem}
Let $G$ be a finite subgroup of a M\"obius group of fractional-linear transformations acting on $\CP^1.$

How to describe $G$-invariant irreducible real algebraic curves on $\CP^1$? 
\end{problem}

Despite that  topological properties of root clustering are those mainly studied there, in spite of future research and computability issues all considerations are made in a real algebraic category, and in cases, where it is inevitable, in a category of semialgebraic spaces. We are not discussong there such important things as borders and singularities of strata, their geometric and metric and convexity-like properties, algorithmic problems.  unconstrained parametric robust stability problems for concrete or general families non-indeterminate polynomials,

The paper is organised as follows. Second section is devoted to a historical review and references on different aspects of the questions involved. 
Third section is devoted to the discussion of various technical definitions and results needed for the development of theory. Stratified filtered real algebraic varieties, their semialgebraic subspaces and group actions on them are discussed there.

Fourth section is devoted to the introduction of different definitions needed for the study of stabilities. Namely, stability theories, which are pointed stratifications on the space of roots of linear polynomial ($\CN$ or $\CP^1$) are introduced in the definition \ref{theories}. Definitions \ref{affDstratif} and \ref{Dstratif} introduce the concept of $D$-stratification. Various refined version of $D$-stratification are discussed, as well as connections to the matrix stabilities and $D$-decomposition construction.

Theorem \ref{classicalpoly} gives a formulation of root-coefficient correspondence in terms of morphisms between filtered stratified real algebraic varieties needed for the study of stability theories.

Theorem \ref{matrixpoly} shows interaction of that correspondence with matrix stability problems via duality between two types of perturbations of polynomial. Finally, Theorem \ref{DdecDstrat}, shows connections between classical concept of $D$-decomposition and more refined $D$-stratification.

Section 5 is devoted to the study of the topological properties of $D$-strata. Topological structure of stratum for a wide class of definitions of stability is examined there. That class contains not only such classical examples as Hurwitz or Schur stability, but any union $LMI$-definable regions, any union of curvilinear polygons or real algebraic curves. Theorem \ref{hott} gives a homotopy type of a stratum and its fundamental group. Betti numbers of $D$-strata are computed in the Theorem \ref{betti}. Proof of these and other results are based on Proposition \ref{unmixing} which allows to reduce questions on the geometry and topology of $D$-strata to the geometry and topology of stability region, it's border and complement to it's closure. All process of these proofs could be done in a semialgebraic category and, therefore, in a constructive, computable way. That makes use of the semialgebraic topology important there.

Propositions \ref{compactconnected} and \ref{poleplacement} gives topological description of $2$ more important cases, namely, the case of connected convex region and the case when stability region is a finite number of points, respectively.

Section 6 describes adjacencies between $D$-strata. Adjacency digraph for a stratified space is defined in Definition \ref{adjdig}.  Each edge of that digraph corresponds to a non-separated pair of stratum. It is proved that the functor of taking adjacency digraph commutes with symmetric product functor (Theorem \ref{adjdigstab}). This leads to the criterium of adjacency between $D$-stratum(Theorem \ref{adjcrit}). Necessary and sufficient conditions for a digraph to be an adjacency graph for a stability theory are given in Propositions \ref{AdjS} and \ref{Adjunderline}.

Section 7 is devoted to the characterization of the most important stability theories in a class of all stability theories. It is proved that if the border of a stability region is an irreducible real algebraic curve, stability theory agrees with different degree-changing perturbations of a polynomial and with transfer to the theory of polynomials with real coefficients and if it is supposed that stability theory ``measures'' small and big roots (i.e. $0$ and $\infty$ cannot both be stable or both be unstable) then it is up to the interchange between stable or unstable regions or taking their union, one of just $3$ most classical theories: Hurwitz stability, Schur stability and hyperbolicity. This is the content of Theorem \ref{standard}. The proof is based on classification of palindromic polynomials by I. Markovsky and S. Rao \cite{MR2008} and high rigidity of the irreducibility condition.

Different possibilities of relaxing conditions are examined in Proposition \ref{orbit} and Proposition \ref{palind}, which leads to a questions on the structure of finite M\"obius group actions.

Paper makes use of different standard notions and results of real algebraic geometry, algebraic topology and category theory. Books by S. Basu, R.Pollack, M.-F. Roy \cite{BPR1998}, A. Hatcher \cite{Hat2002} and  S. MacLane \cite{McL1998}, could be seen as introductions to the subjects, respectively.

I wish to thank Prof. Alexander Engstr\"om, Dr. Cordian Riener, Konstantin Tolmachov, Dmitry Korshunov, Prof. Niels Schwarz, Prof. Manfred Knebusch, Prof. Claus Scheiderer, Eric Wofsey and many other people for useful discussions on the subject. Moreover, i wish to thank prof. Alexander Engstr\"om and dr. Cordian Riener for their hospitality during my quite unusual long-term visit in Aalto university in which part of that work have been done; prof. Markus Schweighofer and Zukunftskolleg staff for their help during my work in University of Konstanz.

I wish also to thank Jenny Curpen and, my mother, Svetlana Gofman for their support.

\section{Historical review}
First root clustering problems ever considered, were probably the problems close to {\it hyperbolicity} , concerning number of (may be, positive or negative) real roots of a polynomial. These problems have appeared already in La G\'eom\'etrie by R. Descartes from 1637, where Descartes rule of signs have been established. Another root clustering problem \--- Hurwitz stability problem, concerning the case of root clustering region being a left half-plane have been studied since the middle of 19-th century, when the first papers in control theory, such as works on stability by J.C. Maxwell \cite{Max1868} and I. Vyshnegradsky have appeared  \cite{Vish1876}. 

Another root clustering problem \---  Schur stability problem, where root clustering region is an open unit-disk have been studied since the work of I.Schur from 1917 \cite{Sch1917}.

Other special root clustering problems, such as sectoral conditions \cite{PS2002} and other sectoral stability regions, unions of disjoint disks \cite{AY1997}, plane curves \cite{Maz1999}, Cassini ovals \cite{MD2014} and many other problems have been considered in different algebraic and algorithmic aspects.

General theory of root clustering  has been developed by E. Jury and S. Gutman since the beginning of 1980's (see \cite{GJ1981} for formulation of the foundations and book \cite{Gut1990}). That theory is based on the study of regions that could be transformed into a Hurwitz stability region by certain polynomial transformations. The explanation of a potential fruitfulness of that transformation-style approach to the study of root clustering problems, could possibly be explained by the recent results of D. Chakrabarti and S.Gorai \cite{CG2015} who proved, in our terms, that any proper holomorphic map between stability regions for an indeterminate polynomials with complex coefficients is induced by proper holomorphic map between clustering regions.

Different algebraic conditions and algorithms for the root and eigenvalue localisation in  $LMI$-regions(i.e. rigidly convex ones \cite{HV2007}) \cite{CG1996,CGA2002}, Ellipsoidal Matrix Inequality Regions ($EMI$-regions) \cite{Peauc2000,PABB2000}, Extended Ellipsoidal Matrix Inequality Regions ($EEMI$-regions) \cite{BM2003}, Boolean combinations of different regions\cite{BHPM2004,BBM2005} have been established since the middle of 1990's.. 

Another notable contribution to root clustering problems was made by J.B. Lasserre in 2004 \cite{Las2004}. He established a criterium for roots of polynomial to be situated in certain semialgebraic region of a complex plane. Lasserre criterium is formulated in terms of moment matrices.

Geometrical ideas in root clustering have appeared for the first time via the parametric approach to robust stability problems, which constitutes an important part of control theory. 

First studies in that line of research could be attributed to I. Vysh\-ne\-grad\-ski.\cite{Vish1876}, who considered different affine sections for the of the spaces of lower-degree real indeterminate polynomials. These ideas have been developed by R.A. Fraser and W.J. Duncan in 1929 \cite{FD1929}. Fundamental ideas of that approach\-- method of $D$-decomposition in have been developed in the end of 1940's in a book by Yu.I. Neimark \cite{Nei1949}. He studied decomposition of a parameter space for a given unconstrained robust stability problem into regions the with the same number of Hurwitz stable and Hurwitz unstable roots. 

Further development of that method was pursued by D. Mitrovic \cite{Mit1959}, D.Siljak \cite{Sil1969}, et al. Higher-dime\-nsional $D$-de\-com\-po\-si\-tion have been studied in the book by S.H.Lehnigk \cite{Leh1966}. In the beginning of 1990's that method became a standard part of any book on parametric approach to robust control \cite{Ack2002,Bar1993,BKC1995}.  

Despite that, knowledge about geometric aspects of $D$-decomposition and geometric structure of a stability region is very limited even for the most important cases, such as $PI$ and $PID$-controllers. Among these results should be noted the theorem from 1998 proved by M.-T. Ho, A. Datta and S.P. Bhattacharyya \cite{HDB1998,HDB2000}, generalised in 2003 by J. Ackermann and D.Kaesbauer \cite{AK2003} which states that stability region of $PID$-controller synthesis problem with fixed proportional gain consists from disjoint union of convex polygons. But, as noted by D.Henrion and M. Sebek \cite{HS2008} even the number of stability regions for $PI$ or $PID$-controller synthesis problem is not known.

Geometric methods for the study of $1$ and $2$-dimensional $D$-decomposition problems, including some estimates on the number of regions of $D$-de\-com\-po\-si\-tion have been developed by E.N. Gryazina and B.T. Polyak in the middle of the 2000's \cite{G2004,GP2006}, their review (together with A.A. Tremba) \cite{GPT2008} provides a good exposition of state-of-the-art $D$-decomposition theory at that time. Around the same time Yu.P.Nikolayev \cite{Nik2002,Nik2004} built several highly non-trivial examples, showing possible complexity of a $D$-decomposition structure even for the $2$-dimensional case. During  2012-2014 \cite{Vas2012a,Vas2012b,Vas2014} author introduced several new tools. Applications of topology of real algebraic varieties and computational real algebraic geometry to the problem gave possibility to estimate the number of regions in an arbitrary dimension and provide a new class of algorithms for the study of $D$-decomposition.

Another set of geometric results could be attributed to the tradition of so-called universal parameter space methods, appeared since the paper by A. Fam and J. Meditch from 1978 \cite{FM1978}.  These results corresponds to the algebro-geometrical and topological study of the space of real indeterminate  (Hurwitz or Schur) stable polynomials from the control-theoretic vewpoint. Among other researchers thinking along that line one can note J.Ackermann \cite{Ack1980}, H.Akyar\cite{Aky2009}, B. Aguirre-Hernandez et al. \cite{AHFAV2012},  G.-R. Duan and M.-Z. Wang \cite{DW1994,DW1996}.

Another important result is G.Meisma's \cite{Mei1995} elementary geometric proof of classic Routh-Hurwitz stability criterium, not using any advanced methods of complex analysis, as in earlier proofs, but only a continuity of dependence of coefficients of polynomials from their roots.  That kind of approach is very similar to the classical approach to the geometry of the $D$-decomposition by Yu.I. Neimark \cite{Nei1949}, where he used quite informal approach of ``movement'' of roots around the complex plane with their ``births'' and ``deaths'' at $\infty.$  Close set of ideas is pursued in a completely algebraic approach to the Routh-Hurwitz theory have been presented in \cite{BPC2006}.

So the {\it root-coefficient correspondence} \--- mapping from the space of roots of a polynomial to it's coefficients appears on the scene.

These results gives hope for a {\it purely geometric} approach to the root clustering and possibility of creating purely geometric theory for unconstrained robust root clustering problems. 
That hope, could of course, be supported by successful use of algebro-geometric methods in closely related area of control, namely, pole placement theory \-- which is consists in clustering eigenvalues of certain affine families of matrices in a prescribed finite subset of a complex plane (see \cite{EG2002} and references therein).

Approach presented here is an attempt to build an unified geometric theory behind unconstrained robust root clustering problems, based on several standard natural geometric constructions instead of different algorithmic approximations and ad-hoc methods. 

In order to discuss that further on, we need to introduce several new actors playing inside a history of proposed concept. 
First of them is an interpretation of the root-coefficient correspondence as a {\it symmetric product} map, i.e. map transforming ordered sets of roots into unordered sets of polynomial coefficients.
As this map is given, essentially, by elementary symmetric polynomials, it is possible to trace it back to the F.Vi\`ete and A.Girard from 16th-17th century(see \cite{Fun1930} for the history of a question), but the concept of symmetric product of topological spaces have been for the first time introduced only at 1931 by K.Borsuk and S.Ulam \cite{BU1931}. Their definition is different from the most common contemporary definition, while the latter one have been introduced first in the series of papers by M.Richardson and P.A. Smith from 1930's\cite{Ric1935,RS1938,Smi1932}. Interpretation of correspondence between roots and coefficients of a polynomial as a symmetric product of topological spaces have been used at least since V.I. Arnold papers on algebraic version of 13th Hilbert Problem \cite{Arn1970}, but only in 2012 B. Aguirre-Hernandez, J.L. Cisneros-Molina, M.-E. Frias-Armenta \cite{AHCMFA2012} introduced that interpretation into the domain of control theory and stability problems. Symmetric product construction have been used there in order to show that the spaces of aperiodic monic Schur or Hurwitz polynomials with real coefficients are contractible.

Another important player in our exposition is different {\it stratifications of the space of polynomial coefficients} induced by the structure of the root spaces. Although our main goal is to develop theory of the stratifications for the root clustering problems, historically main attention have been pointed to the stratification of the space of indeterminate polynomials produced by the multiplicities of roots. That subject known under different names (one of the best-known of them is {\it coincident root loci}), although been in an attention since the paper of D.Hilbert from 1887 \cite{Hilb1887} still produces new questions and results in algebraic geometry and singularity theory \cite{Kh1987,KS1990,Chi2003,Chi2004,Kat2003,Kur2012,SL2016,Mikh2015,Nap1998}. Computations of homologies for different strata of multiplicity stratification of the space of polynomials \cite{Arn1970,Arn1989,GKY1999,Kam2008,SW1996} motivated, again, by V.I. Arnold ideas and connected with an applications of general theory of topology of complements to discriminants created in the beginning of 1990's by V.A. Vassiliev \cite{VVas1994} constitutes another important topic.

From the other side singularities of stability borders, especially in connection with hyperbolic polynomials have been studied by V.I. Arnol'd, B.Shapiro, A.D. Vainstein \cite{Arn1972,Arn1986,VS1990}. As noted to author by O.N. Kirillov, some of the singularity-theoretic phenomena arising there have been discovered earlier by physicist O. Bottema \cite{Bot1956}.
For the case of Hurwitz polynomials singularities of the border have been studied by L.V. Levantovskii\cite{Lev1982}. The book by V.P. Kostov \cite{Kos2011} contains a study of stratification of the space of hyperbolic polynomials given by different multiplicities of roots. Book by A.P. Seyranian, A.A. Mailybaev\cite{SM2004} and later papers by O.N. Kirillov \cite{HK2010,KO2013} presents an applied view on singularities of the stability border. Several results \cite{Arn1986,Che1986,DG1994,Meg1992,Obr1963} is known about the convexity-like properties of the space of hyperbolic polynomials. From the other perspective geometric and convexity-like properties of Schur stability region(known also as a symmetrized polydisc) have seen a considerable attention from complex analysts(e.g. \cite{EZ2005,Cos2005,Nik2006,NPZ2008,NTT2016}), having in mind both its importance for geometric function theory and connections with $\mu$-synthesis problems. 

During the last years D.Chakrabarti et al. are studying geometry of symmetric products of regions of complex plane from the complex-analytic point of view.\cite{CCGL2014,CG2015}.  

Finally, J.Borcea and B.Shapiro \cite{BS2004}  classified  stratifications of $1$-dimensional affine families induced by the multiplicity stratification on the space of polynomials with real coefficients. Analogous results for stratifications induced by stabilities and higher dimension of a family, could be seen among the ultimate goals of our approach.

\section{Prerequisites: stratified filtered real algebraic varieties and symmetric products}
Here we introduce basic tools for the study of the stability problems. Filtered real algebraic varieties and their semialgebraic subsets defined below are used to capture the natural structure of filtration of the space of polynomials and binary forms by degree, matrices filtered by size, sequences of roots filtered by their length.
\begin{definition}
A {\it filtered real algebraic variety} $L=\{(L_i,\lambda_i),i\in\N\}$ is a (possibly infinite) sequence of closed embeddings of real algebraic varieties $$L_0\overset{\lambda_0}{\rightarrow} L_1\overset{\lambda_1}{\rightarrow} \ldots.$$

Morphism between filtered real algebraic varieties $\varphi\colon L\to R$ is a sequence of morphisms $\varphi_i\colon L_i\to R_i$ that commutes with embeddings.
\end{definition}

That definition is parallel to the I.R. Shafarevich's definition of an in\-fi\-ni\-te-di\-men\-sio\-nal algebraic variety \cite{Sha1981}.\footnote{Since in that paper we are going to limit ourselves with bounded-dimensional considerations we do not need here any more abstract formalism for working with semialgebraic sets and real algebraic varieties in essentially in\-fi\-ni\-te-di\-men\-sio\-nal setting that could be provided by a development of $ind$-scheme and $ind$-group theory for the category of N.Schwarz (inverse) real closed spaces \cite{Sch1989,Sch1991} or any other formalism in semialgebraic geometry, which constitutes an open problem.

One of the motivation for that open problem is that the stability problem is often an affine subspace in the space of all polynomials. Consider the space of all stability problems with fixed ``structure of the controller''. It could be seen as is a Grassmann variety, filtered up to infinite dimension. That G variety is once again a stratified into ($ind$-)semialgebraic subsets representing stability problems of fixed topological structure of their stratified parameter space. Geometry (including real algebraic and metric geometry) of these subsets actually is an exact geometry of robustness for parametric stability problems. That line of thought on robustness is parallel to the ideas of algebraic statistics.}

As an illustrative definition for an abstract semialgebraic set we can take the following definition:
\begin{definition}
Let $L$ be a real algebraic variety and let $\varphi\colon L\to \R^k$ be an embedding of $L$ into real affine space. $S$ is {\it a semialgebraic subset} of $L$ if $\varphi(S)$ is semialgebraic.
\end{definition}
From results \cite{DK1981} it is easy to see that:
\begin{lemma}
If $S\subset L$ is a semialgebraic subset with respect to embedding $\varphi$, then it is semialgebraic subset with respect to any embedding.
\end{lemma}
For the thorough study on abstract semialgebraic sets (i.e. se\-mi\-al\-ge\-bra\-ic spa\-ces) author refers reader to the sequence of books and papers by H.Delfs and M.Knebusch \cite{DK1981a,DK1981,DK1985,Kne1989,Del1991}. However we are trying to evade use of that language there as much as possible and most of the paper could be understood without it.
\begin{definition}
 Let $L$ be a filtered real algebraic variety. A sequence of semialgebraic subsets $S_0\subseteq S_1\subseteq\ldots;\quad$  $S_i\subseteq L_i$ is called {\it a filtered semialgebraic subset of $L.$}
\end{definition}

\begin{definition}
 Pair $(L,S),$ $L=\sqcup_{s\in S}S,$ where $L$ is real algebraic variety and $S$ is a set of semialgebraic subsets of $L$ is called {\it stratified real algebraic variety.}
\end{definition}
\begin{definition}
 Let $L$ be a filtered real algebraic variety, and let all $L_i$ be equipped with such a stratification $S_i$ that $$\forall s\in S_i\quad\lambda_i(s)\subset \widetilde{s}\in S_{i+1},\qquad \widetilde{s}\cap\lambda_i(S_i)=\lambda_i(s)$$ then $L$ is {\it filtered stratified real algebraic variety.} A filtered real algebraic variety could be seen as a stratified filtered real algebraic variety with trivial stratification.

 Let $(L,S)=\{(L_i,\lambda_i,S_i), i\in\N\},(T,Q)=\{(T_i,\tau_i,Q_i), i\in\N\}$ be filtered stratified real algebraic varieties. The sequence of morphisms $\varphi=\{\varphi_i\colon L_i\to T_i, i\in\N\}$ is the {\it morphism of filtered stratified real algebraic varieties} if for each $s\in S_i$ there exists $q\in Q_i$ such that $\varphi(s)\subseteq q,$ for each $s,t\in S_i$ $\varphi(s)=\varphi(t)$ or $\varphi(s)\cap\varphi(t)=\varnothing$ and $\forall i\in\N\quad \varphi_i\circ\tau_i=\lambda_i\circ\varphi_{i+1}.$
 
 If $(L,S)$ is a filtered stratified real algebraic variety, then there exist a canonical forgetful morphism of real stratified algebraic varieties $\lambda^{id}=\{id_{L_i}\}.$
\end{definition}

\begin{definition}
 Let $G_0\subseteq G_1\subseteq \ldots = G$ be a filtered algebraic group(i.e. filtration of algebraic groups by sequence of closed embeddings) and let $(L,S)$ be a filtered (stratified) real algebraic variety.
 
Define {\it a filtered action $\gamma$ of $G$ on $L$} as sequence of actions $\gamma_i\colon G_i\to Aut(L_i)$ $G_i$ on $L_i$ that commutes with embeddings.

$\gamma$ {\it respects stratification $S$} if for each $g\in G_i$ and each $s\in S_i$ there exists $\widetilde{s}\in S_i$ such that  $\gamma_i(g)s=\widetilde{s}.$
\end{definition}
\begin{proposition}\label{infprod}
Let $(R,S)$ be a stratified real algebraic variety with marked point $e.$ Sequence of morphisms  
$$R^0=\{e\}\overset{\varphi_0}{\to}R\overset{\varphi_1}{\to}R^2\to\ldots,\quad\varphi_i\colon (r_1,\ldots,r_i)\mapsto (r_1,\ldots r_i,e)$$ with stratifications produced by componentwise products of $S$-stratum is a  filtered stratified algebraic variety $(R,S)^{\infty}$ \-- {\it infinite product} of $(R,S).$
\end{proposition}
\begin{proof}
 Take such $\widehat{s}\in S$ that $e\in S.$ Then $\varphi_i(s_1\times\ldots\times s_i)\subseteq s_1\times\ldots\times s_i\times\widehat{s}.$ Moreover $s_1\times\ldots\times s_i\times\widehat{s}\cap\varphi_i(R^i)=s_1\times\ldots\times s_i\times\{e\}.$
\end{proof}
Now we are able to define spaces and groups that will be main players in our exposition.
\begin{example}
\begin{enumerate}
\item $U_{0}\subset U_1\subset\ldots \subset U_{i}\subset$ is a  filtered space $U$ of all polynomials with complex coefficients. Here $U_i$ is a $(2i+2)$-dimensional space of polynomials degree less than $i$ with embeddings given by $x\mapsto (x,0+0i).$ It could be also interpreted as space of all homogeneous binary forms $f(x,y)$ with embeddings given by $f(x,y)\mapsto f(x,y)y.$ 
\item $V_{0}\subset V_1\subset\ldots \subset V_{i}\subset$ is a filtered space $V$ of all monic polynomials with complex coefficients. Here $V_i$ is a $2i$-dimensional space of polynomials degree less than $i$ with embeddings given by $x\mapsto (0+0i,x).$
\item $\CN^{\infty}$ \-- filtered space of complex sequences with finite number of non-zero elements (isomorphic to $V$);
\item $(\CP^1)^{\infty}$ \-- filtered space of finite sequences of points from $\CP^1$ with finite number of non-infinity elements;
\item $\CP^{\infty}$\-- is a filtered real algebraic variety givean by sequence of morphisms that could be written in complex homogeneous coordinates as $[x_0:\ldots:x_k]\mapsto[x_0:\ldots:x_k:0]$
\item $Mat(\CN,\infty)$ \-- filtered space of square matrices with finite number of non-zero entries;
\item $Gl(\CN,\infty)$ \-- filtered algebraic group of invertible transformations of $\CN^{\infty};$
\item $\Sigma^{\infty}$ \-- infinite symmetric group (permutations with finite number of non-stable points);
\end{enumerate}
\end{example}

In what follows we study group actions on filtered stratified real algebraic varieties and define a symmetric product construction.
\begin{definition}
Let $R$ be a semialgebraic space. Let $\Sigma_n$ be a symmetric group acting on $R^n$ by permutations of coordinates. 
Then $n$-th symmetric product of a semialgebraic space $R$ is a quotient of $R^n$ by an action of symmetric group $\Sigma_n.$ It is denoted by $R^{(n)}.$
\end{definition}

\begin{proposition}\label{semialg}
$n$-th symmetric product $R^{(n)}$ of semialgebraic space $R$ is a semialgebraic space. Points  of $R^{(n)}$ could be identified with cardinality $n$ multisubsets of $R.$
\end{proposition}
\begin{proof}
Denote by $E\subset R^n\times R^n$ an equivalence relation on $R^n.$ 
Since $\Sigma_n$ is a finite group, projection map $\pi$ from graph of equivalence relation $E$ to $R^n$ has finite fibers. Hence $\pi$ is proper. Hence, by Theorem 1.4 \cite{Bru1987}, quotient space in the topological category is a quotient space in the category of semialgebraic spaces. Thus, its elements could be identified with equivalence classes of $\Sigma_n$-action, which are in $1-1$ correspondence with unordered subsets having the same multiplicity of each element. 
\end{proof}
\begin{definition}
{\it Infinite symmetric product} of a real algebraic variety $R$ with marked point $e$ is a filtered real algebraic variety $R^{(\infty)}$ given as sequence of quotients  defined by filtered action by permutations of filtered group $\Sigma^{\infty}=\Sigma_0\subseteq\Sigma_1\subset\Sigma_2\subset\ldots$ on infinite product of $R.$ 

If embedding defining a filtration on product is denoted $\lambda_i, i\in\N$ then embeddings defining filtration on symmetric product will be denoted $\lambda_{(i)}.$

$n$-th member of filtration is an $n$-th symmetric product of $R.$ It is denoted by $R^{(n)}.$ 
\end{definition}

Note that, in general, symmetric products of real algebraic varieties could be not real algebraic varieties, but only semialgebraic spaces(abstract semialgebraic sets), nevertheless symmetric product of smooth complex algebraic curve is a smooth complex algebraic variety \cite{McD1962}, moreover, we can state an important proposition, which, actually forms a basis for all our consequent considerations. Some form of that proposition appears for the first time probably in a paper by S.D.Liao from 1954 \cite{Lia1954}.  
\begin{proposition}\label{proj}
Infinite symmetric products of $\CN$ and of $\CP^1$ are filtered real algebraic varieties.
$\CN^{(\infty)}\cong V,$ while $(\CP^1)^{\infty}\cong \CP^{\infty}\cong {\bf P}(U).$

First isomorphism is given by the correspondence between roots and coefficients of monic polynomials, while the second is provided by the decomposition of binary forms into the product of linear one's.
\end{proposition}
\begin{proof}
Let us identify $\CP^n$ with the space of homogeneous binary forms $f=\sum_{i=0}^n a_ix^iy^{n-i}$ of degree $n$ defined up to a constant multiple. Each such form could be uniquely (up to reordering) decomposed into the product of linear forms (defined up to a constant multiple ) $\prod_{i=1}^n(\alpha_ix-\beta_iy).$

This space of linear forms could be interpreted as $(\CP^1)^{n}.$ Thus one can see an initial binary form as an equivalence class of $\Sigma_n$-action on $(\CP^1)^{n},$ which gives a symmetric product morphism. This map is defined by polynomials in homogeneous coordinates, thus it is a morphism of complex algebraic varieties, and, hence it is a morphism of real algebraic varieties.

Take some product of linear homogeneous binary forms $\prod_{i=1}^n(\alpha_ix-\beta_iy).$ Its multiplication by $y$ defines a closed embedding of $(\CP^1)^n$ into $(\CP^1)^{n+1}.$  Polynomials defining quotient map are homogeneous, therefore embedding commutes with symmetric product morphisms.

So, we have a morphism of filtered real algebraic varieties.

Proof for the case of $\CN$ could be done either by analogy or by proceeding to the space of linear homogeneous binary forms with an $x$-coeffient non-equal to zero and taking a corresponding filtered real algebraic subvarieties. 
\end{proof}

\begin{proposition}\label{stratif}
Let $(R,S)$ be a stratified real algebraic variety with marked point $\{e\}$. Take some $n\in\N\cup\{\infty\}$ such that for each $m\leq n$ $R^{(n)}$ is a real algebraic variety.
 
Denote by $\kappa_m\colon R^m\to R^{(m)}$ the symmetric product morphism
 
Then $R^{(n)}$ is a stratified filtered real algebraic variety with stratification $S^{(n)}$ given by sets $\{k_i|i\in S\},$ that consists of such points $x\in R^{(n)}$ that for each $i\in S$ is number of points from $i$ in coordinate projection of $\kappa_n^{-1}(x)$ is $k_i.$
Filtration of $R^{(n)}$ is given by the embeddings of multisets $$\rho_i\colon R^{(i)}\to R^{(i+1)},\quad \{x_1,\ldots,x_i\}\mapsto\{x_1,\ldots,x_i,e\}.$$
\end{proposition}
\begin{proof}
$\kappa_m$ is a quotient by a finite group action, hence it is closed map. Therefore, by Proposition \ref{infprod} , $R^{(n)}$ is filtered and $S$ induces a stratification of $R^{n}.$ 
First note that for each $0< i\leq n,$ and $s_0,s_1\in S$ either $\kappa_i(s_0)=\kappa_i(s_1)$ or $\kappa_i(s_0)\cap\kappa_i(s_1)=\varnothing.$ 

Therefore, $(R^{(i)},\kappa_i(S^i))$ is stratified. Take $s\in S^{i}$ and take such $\widetilde{s}\in S^{i+1}$ that $\lambda_{i}(s)\subseteq\widetilde{s}.$ 
Let $x\in \lambda_{(i)}(\kappa_{i+1}(S))\cap\kappa(\widetilde{s}).$ Then there exists $\sigma,\mu\in\Sigma_{n+1}$ $\sigma x\in\widetilde{s},$ therefore $x\in\kappa_{i+1}(\lambda_i(s))=\lambda_{(i)}(\kappa_i(s)).$ Hence stratification agrees with filtration.
\end{proof}
\begin{definition}
Let $R$ be a real algebraic variety. Then its symmetric product $R^{(n)}, n\in\N\cup\{\infty\}$ admits a canonical stratification $\mu(R^{(n)}),$ with stratum parametrised by partitions of $n,$ namely,                
$x\in(m_1,\ldots,m_k), \sum_iim_i=n$ iff preimage of $x$ under canonical morphism $\kappa_n\colon R^n\to R^{(n)}$ have exactly $m_i$ components of multiplicity $i.$

Strata of $\mu(R^{(n)})$ corresponding to a partition $\lambda$ will be denoted $\lambda_{\mu}^R.$ In case of $R$ equal to $\CP^1$ upper index will be omitted.

Note that this stratification does not agree with filtrations on $R^{(n)}$ produced by coordinatewise embeddings, 

If $(R,S)$ is a stratified real algebraic variety, then stratification of $R^{(n)}$ given by intersections of stratum from $S^{(n)}$ with $Mult(R^{(n)})$ will be denoted as $\widehat{S^{(n)}}.$
\end{definition}
\begin{proposition}
Let $R$ be a real algebraic variety with marked point $e$. Suppose that for each $m\leq n,$ $R^{(m)}$ is real algebraic variety. Let $$\rho_i\colon R^{(i)}\to R^{(i+1)},\quad\{x_1,\ldots x_i\}\mapsto\{x_1,\ldots x_i, e\}.$$

Let $\eta$ be a partition of $n.$ 
Then, $$\eta_{\mu}\cap\rho_{n-1}(R^{(n-1)})\subset\bigcup_{\begin{array}{l}\lambda\vdash n-1\\\lambda<\eta\end{array}}\rho_{n-1}(\lambda_{\mu})$$ and for each $\lambda\vdash n-1, \lambda<\eta$ holds $\eta_{\mu}\cap\rho_{n-1}(\lambda_{\mu})\neq\varnothing.$

If $\lambda$ differs from $\eta$ on some component $k$ of length $r$ then 
$$\eta_{\mu}\cap\rho_{n-1}(\lambda_{\mu})=\circ_{i=n}^{n-r}\rho_i((\lambda\setminus k)_\mu\setminus\rho_{n-r-1}(R^{(n-r-1)}).$$ 
\end{proposition}
\begin{proof}
Note that if $\lambda \not < \eta $ then $\rho_{n-1}(\lambda_{\mu})$ have some position greater then in $\lambda,$ but $\rho_{i-1}$ by definition is (non-strictly) monotonic on multiplicities.

Take some $\lambda\vdash n-1, \lambda<\eta.$ That means that $\lambda$ is less then $\eta$ on strictly one position $k$ of length $l$. Take a point $q$ from $\lambda_{\mu}$ with multiplicity $l$ on $e.$ Then $\rho_{i-1}(q)\in\eta_{\mu}.$

Note that for each $i<n$ $\rho_i$ adds the point $e$ with multiplicity $1$ to each element. Thus points not belonging the image of $\rho$ are exactly multisets not containing $e.$
This gives the last claim. 
\end{proof}
Finally, we concentrate on different variants of stratifications induced by the symmetric product operation. These stratifications are crucial for the study of singularities of stratified varieties and for the description of the local structure of degree change. 
\begin{definition}
 Let $S, T$ be stratifications of real algebraic variety $R.$ We will write $T\preccurlyeq S$ iff for each $s\in S$ there exists $\tau\subseteq T$ such that $s=\cup_{t\in\tau}t.$
 If $R$ is filtered then $T\preccurlyeq S$ iff for each $i\in\N$ $T_i\preccurlyeq S_i.$

 Relation $\preccurlyeq$ defines a partial order on a set of stratifications of $R.$ 
\end{definition}
\begin{lemma}\label{localstratum}
 Let $(R,S)$ be a stratified filtered real algebraic va\-rie\-ty gi\-ven by the se\-que\-nce $(R_i,S_i)$ of filtered real al\-ge\-bra\-ic va\-rie\-ti\-es and clo\-sed em\-bed\-dings $\lambda_i\colon R_i\to R_{i+1}.$
 
 Then there exist unique maximal stratification $\underline{S}\preccurlyeq S$ such that $(R,\underline{S})$ is a filtered stratified algebraic variety and for each $s\in\underline{S}_i,\,$ $\lambda_i(s)=t\in \underline{S}_{i+1}.$  
 
 $\underline{S}$ could be inductively defined by the following way:
 
 $$\underline{S}_0=S_0,\quad \underline{S}_{i+1}=\{s\in S|s\setminus \lambda_i(R_i)\}\cup\{t\in\underline{S}_i|\lambda_i(t)\}).$$
\end{lemma}
\begin{proof}

Note that if for some stratification $T$ of $R$ for each $i\in\N\, s\in T_i,\,$ $\lambda_i(s)=t\in T_{i+1}$ then $(R,T)$ is a filtered stratified real algebraic variety.  

Hence $\underline{S}$ defines filtered stratified algebraic variety $(R,\underline{S})$ that satisfies $s\in\underline{S}_i,\,$ $\lambda_i(s)=t\in \underline{S}_{i+1}.$ 
 
Suppose filtered stratified real algebraic variety $(R,T),\,T\preccurlyeq S$ satisfies our assumptions. Note that $T_0\preccurlyeq S_0=\underline{S}_0.$
Proceed by induction, namely if $T_i\preccurlyeq \underline{S}_i$ then $\lambda_i(T_i)\preccurlyeq\lambda_i(\underline{S}_i),$ as $\lambda_i$ is an embedding. 
Assumptions give $\lambda_i(T_i)\subseteq T_{i+1},\, \lambda_i(\underline{S}_i)\subseteq\underline{S}_{i+1}.$
But $S_{i+1}|_{R_{i+1}\setminus \lambda_i(R_i)}\subseteq\underline{S}_{i+1}.$ Hence $T_{i+1}\preccurlyeq\underline{S}_{i+1}.$
\end{proof}
Each of these stratifications plays a special role in root clustering. Namely, $S$ deals with the global structure of a clustering, $\underline{S}$ reflects local structure of degree change. $\widehat{S}$ reflects the structure of borders, corners and singularities of stratum. 

One can also easily prove the following lemma.
\begin{lemma}\label{hierarchyofstratum}
Let $(R,S)$ be a stratified real algebraic variety.
Then for each $R_i$ there exists $4$ canonical stratifications with order diagram as below:
\begin{center}
\begin{tikzcd}
&\widehat{\underline{S^{(i)}}}      \arrow[Preccurlyeq]{dr}{}\arrow[Preccurlyeq]{dl}{} &\\
  \underline{S^{(i)}} \arrow[Preccurlyeq]{dr}{}&& \widehat{S^{(i)}}\arrow[Preccurlyeq]{dl}{}\\
&S^{(i)}&
\end{tikzcd}
\end{center}
\end{lemma}

Our goal here is to study these stratifications for the case symmetric powers of spaces $\CN$ and $\CP^{1}.$ Stratification $S^{(n)}$ is a main player of our exposition, while $\underline{S^{(n)}}$ is important in the local study of a filtered structure, while geometry of $\widehat{S^{(n)}}$ and $\widehat{\underline{S^{(n)}}}$ is mainly left for the future research,
as it's mainly connected with singularities, borders and higher-codimensional corners of $S^{(n)}$-strata.
\begin{proposition}\label{isolocal}
Let $(R,S)$ be a stratified real algebraic variety with marked point $e,$ such that for each $m<n$ $R^{(n)}$ is an algebraic variety.
Denote by $T$ a stratification $\{s\setminus \{e\}|s\in S\}\cup\{\{e\}\}.$

Then  $T^{(m)}=\underline{S^{(m)}}$
\end{proposition}
\begin{proof}

Take $m=0.$ Then $T^{(0)}=S^{(0)}=\underline{S^{(0)}}=\{\{e\}\}.$ 

Proceed by induction. Suppose that $T^{(i)}=\underline{S^{(i)}}.$ Using Proposition \ref{stratif} we get 
$$\lambda_i(T^{(i)})=\{\{t\cup\{e\}|t\in\tau\}|\tau\in T^{(i)}\}.$$  $\{e\}$ is a stratum of $T.$ Hence by Proposition \ref{stratif} $\lambda_i(T^{(i)})=\lambda_i(\underline{S^{(i)}})\subseteq T^{(i+1)}.$ Moreover, note that if $e\in t\in R^{(i+1)}$ then $t=\lambda_i(t\setminus\{e\})\in\lambda_i(R_i).$ Take $\tau\in T^{(i+1)}\setminus\lambda_i(\underline{S^{(i)}}).$ $\tau$ is a stratum with multiplicity $0$ on $e$ and some multiplicities $\{k_i, i\in T\}$ on other $T$-stratum, thus there exist unique $\sigma\in S^{(i+1)}$ such that $$\sigma=\tau\cup\cup_{j\in I\subset \underline{S^{(i)}}}\lambda_i(j).$$ Comparison with the construction of $\underline{S}$ (Lemma \ref{localstratum}) completes the proof.
\end{proof}

\section{Stability theories}
\begin{definition}\label{theories}
  {\it Stability theory} is a triple $S=(\CP^1,\Omega,\infty).$
  Here $\CP^1$ is considered as real algebraic variety with fixed semialgebraic subset $\Omega$ and marked point $\infty.$ 
  That triple will be considered as stratified real algebraic variety with canonical stratification $$Str(S)=\{\Omega=\Omega_s, \overline{\Omega}\setminus\Omega=\Omega_{ss}, \CP^1\setminus\overline{\Omega}=\Omega_{un}\}.$$ 
  Closure is understood as closure in euclidean topology.
  
  Analogically, {\it monic stability theory} is a stratified real algebraic variety with marked point, considered as a triple $T=(\CN,\Omega,0).$ It's canonical stratification $Str(T)$ is defined by the same way.\footnote{Close theory could be actually developed for any smooth Riemann surface. Here, in order to be closer to motivations we prefer to use just 2 concrete examples.}
  
  Denote by $Str(S)^{con}$ a refinement of $Str(S)$ that consists in decomposition of each stratum into connected components.
  
  The set of connected components of $Str(S)$ stratum $\Omega_i$ will be denoted by $\Omega_{i}^{con}.$
  
  Different examples of stability theories are shown on the Figure \ref{stabreg}
 
\end{definition}

\begin{definition}
Let $p\in\R[i][x]$ be a polynomial. Call root $r$ of $p$ $\Omega$-stable if $p\in \Omega,$ call it $\Omega$-semistable if $p\in\overline{\Omega}\setminus\Omega,$ where closure is considered to be euclidean.
Otherwise call it $\Omega$-unstable.

Each polynomial $p$ has an {\it $\Omega$-stability} index defined as a triple $(r_s,r_{ss},r_{un}),$ $r_s+r_{ss}+r_{un}=\deg p,$ with $r_i$ being a number of roots in $i$-th region of $Str((\CP^1,\Omega,\infty)).$ 
Zero polynomial, by definition has degree continuum, with a corresponding stability index $(|\Omega|,|\overline{\Omega}\setminus\Omega|,\CN\setminus|\overline{\Omega}|).$
\end{definition}

\begin{definition}\label{affDstratif}
Define an affine $D$-stratification $D^{n}_S$ of $U^{n}$ as a most rude decomposition of $U^{n}$ into regions with the same stability index relative to stability theory $S=(\CP^1,\Omega,{\infty})$.

Let $(k,l,m)^{aff}_{S}$ be a stratum with stability index $(k,l,m).$ Denote corresponding stratification of $U$ as $D^{aff}_S=\cup_{i\in\N}D^{n}.$

$D$-stratification for the space $V$ and monic stability theories is defined by the same way.
\end{definition}
%

\begin{theorem}[Root-coefficient correspondence]\label{classicalpoly}
Let us fix some stability theory $S$ with stability set $\Omega.$

Then all morphisms in the following diagram below are morphisms of stratified filtered real algebraic varieties
$$
S^{\infty}\overset{\kappa}{\twoheadrightarrow}S^{(\infty)}\xrightarrow{\sim}\CP^{\infty}=({\bf P}(U),{\bf P}(D_S))\twoheadleftarrow (U,D_S)\setminus\{0\}\hookrightarrow (U,D_S)
$$ 
\end{theorem}
\begin{proof}
Stability theory is a stratified real algebraic variety with underlying space $\CP^1,$ thus by Proposition \ref{proj} and \ref{stratif} we have a diagram of filtered real algebraic varieties.

Let us take a non-zero polynomial $p(z)$ with roots $\alpha_1,\ldots,\alpha_r$ from $U_n.$ It can be represented as binary form $y^np(\frac{x}{y})$ decomposable into product of linear forms $c(x-\alpha_1 y)\ldots(x-\alpha_{r}y)y^{n-r},\quad c\in\CN.$ Proposition \ref{proj} ensures that the strata of ${\bf P}(D^{n}_S)$ and of $S^{(n)}$ coincide.
\end{proof}

\begin{definition}\label{Dstratif}
Let $S$ be a stability theory. Then stratification on $S^{(\infty)}$ induced by $S$ will be called $D$-stratification.

Stratum with stability index $(k,l,m)$ is denoted $(k,l,m)_S.$
\end{definition}
\begin{definition}
Take some stability theory $S$ with stability set $\Omega.$ Monic stability theory $S_m$ with stability set $\frac{1}{\Omega}\setminus\{\infty\}$ is called {\it dual} to $S.$

Stability theory $T_s$ dual to the monic stability theory $T$ with stability set $\Omega$ is a stability theory, which stability set on finite points is defined as $\frac{1}{\Omega},$ and $\infty$ is semistable in $S$ iff $\Omega_{ss}$ is unbounded, while in the other cases $\infty$ belongs to the same stratum as it's sufficiently small neighborhoods. 
\end{definition}
Notion of dual stability theories gives a possibility to formulate a matrix analogue of root-coefficient correspondence and find its connection to a polynomial one, which is given by an action of the inversion on a complex projective line. 

\begin{theorem}[Matrix-polynomial duality]\label{matrixpoly}
Let us fix some stability theory $S$ with stability set $\Omega.$

Consider $Mat(\CN,\infty)$ as a space stratified by the stability theory $S_m$ into sets with the fixed number of eigenvalues belonging to the same stratum of $S_m.$

Let $\chi$ be quotient map under a filtered action of $Gl(\CN,\infty)$ given by coefficients of characteristic polynomial. 

Let  $\pi$ be a projectivisation morphism,
$diag$ be an embedding of diagonal matrices, $inv$\--- inversion and let $\iota$ be a tautological embedding of filtered spaces with monic polynomials being an affine chart for the projective space of binary forms.

Then the following diagram is commutative in category of filtered stratified real algebraic varieties:
$$
\xymatrix{
&&S^{\infty}\ar[r]^{\kappa}&S^{(\infty)}&(U,D^{aff}_S)\setminus\{0\}\ar[l]^{\pi}\\
S_m^{\infty}\ar[urr]^{inv^{\infty}}\ar[drr]^{diag}\\
&&Mat(\CN,\infty)\ar[r]^{\chi}&(V,D_{S_m})\ar[uu]^{\iota}
}
$$
\end{theorem}
\begin{proof}
Take some $(s_1,\ldots,s_n)=s\in S_m^n.$ $\kappa(inv(s))$ is a sequence of coefficients of homogeneous binary form $\prod_{j=1}^n(x-s_j^{-1}y)$ defined up to non-zero complex constant multiple.
The other way around $\pi(\chi(diag(s)))$ is a sequence of coefficients of a polynomial $\prod_{j=1}^n(y-s_j).$
Proceeding to a binary form and taking a constant multiple $(-1)^n\prod_{j=1}^ns_j^{-1}$ we get the same sequence.

Thus diagram is commutative for any fixed $n\in\N$ as a diagram in a category of filtered real varieties. Note that the polynomial $\prod_{j=1}^n(x-s_j^{-1})$ have the same stability index relative to the stability theory $S$ as the stability index of $\prod_{j=1}^n(y-s_j)$ relative to the stability theory $S_m.$ This produces a morphism of stratified spaces.
\end{proof}

The essential sense of the duality is a correspondence between two types of degree-changing deformations of a polynomial, namely
$$
\begin{array}{l}
\text{Polynomials and stability theories:}\\ 
\qquad\qquad\qquad\qquad\qquad\qquad\qquad a_nz^n+\ldots +a_0\mapsto \epsilon z^{n+1}+a_nz^n+\ldots+a_0\\
\text{Matrices and monic stability theories:} \\
\qquad\qquad\qquad\qquad\qquad\qquad\qquad a_nz^n+\ldots +a_0\mapsto z(a_nz^n+\ldots +a_0)+\epsilon.
\end{array}
$$

Now we are able to build a connection between $D$-stratifications introduced here and classical concept of $D$-decomposition for a robust stability problem.

Namely, the following result is straightforward:
\begin{theorem}[$D$-stratification and $D$-decomposition]\label{DdecDstrat}
\phantom{A}
\begin{enumerate}
\item[]
 \item 
 Let $S$ be stability theory with stability set $\Omega.$ 
 Let $h_0+h_1f_1(z)+\ldots+h_rf_l(z)$ be an affine polynomial family of degree $n$ on $z.$ 
 
 That family could be seen as a morphism $\varphi\colon \R^r\to U_n.$
 Then regions of $D$-decomposition of parameter space $(h_0,\ldots,h_r)$ are connected components of $$\varphi^{-1}((k,0,n-k)^{aff}_S\cap Im\,\varphi),\quad k=0,\ldots,n$$
 Border of $D$-decomposition is $$\varphi^{-1}((\cup_{i=0}^n\cup_{k=0}^{n-i}(k,i,n-k-i))_S^{aff})\cap Im\,\varphi)$$
 \item 
  Let $S$ be a monic stability theory with stability set $\Omega.$ 

 Let $h_0A_0+\ldots+h_rA_r$ be a family of $n\times n$ matrices.
 
 That family could be seen as a morphism $\varphi\colon \R^r\to Mat(\CN,\infty).$
 Then regions of $D$-decomposition of parameter space $(h_0,\ldots,h_r)$ are connected components of $$(\chi\circ\varphi)^{-1}((k,0,n-k)^{aff}_S\cap Im\,(\chi\circ\varphi)),\quad k=0,\ldots,n.$$
 Border of $D$-decomposition is $$(\chi\circ\varphi)^{-1}((\cup_{i=0}^n\cup_{k=0}^{n-i}(k,i,n-k-i)_S)\cap Im\,(\chi\circ\varphi)).$$
\end{enumerate}

\end{theorem}

\section{Topology of a $D$-stratum}
Following lemma allows to reduce topology of $D$-strata to the topology of strata of the stability theory, hence forming a basis for the consequent considerations.

\begin{lemma}\label{dispower}
Let $S$ be a connected semialgebraic space. Suppose that $S=\cup_{i=1}^mS_i$ is a union of semialgebraic subspaces such that $S_i\cap S_j=\varnothing$. Suppose that for each $i\in\{1,\ldots,m\}$ $S_i$ is either open or closed.
Take some $\lambda=(\lambda_1,\ldots,\lambda_m)\in\N^m,$ such that $\sum_{i=1}^m\lambda_i=k.$

Define a map 
$$\varphi\colon \prod_{i=1}^mS_i^{(\lambda_i)}\to S^{(k)},\quad(\{s^i_1,\ldots s^i_{\lambda_i}\})_{i\in\{1,\ldots,m\}}\mapsto \cup_{i=1}^m\{s^i_1,\ldots s^i_{\lambda_i}\}.$$

Then $\varphi$ is a semialgebraic homeomorphism onto image.

\end{lemma}
\begin{proof}

Since $S_i$ are pairwise disjoint, $\varphi$ is injective. 
Denote $\prod_{i=1}^mS_i^{\lambda_i}$ as $Q$.

Since $\kappa_k$ is continuous, $\kappa_{k}|_{\prod_{i=1}^mS_i^{\lambda_i}}$ is also continuous.

Let us prove that $\kappa_{k}|_{Q}$ is closed. Note that $\kappa_k|_{\kappa_{k}^{-1}(Im\,\varphi)}$ is closed as a quotient map by a finite group action.
It is easy to see that $\Sigma_kQ=\kappa_{k}^{-1}(Im\,\varphi)$ and, moreover for each $\sigma\in \Sigma_k$ either $\sigma Q=Q$ or $\sigma(Q)\cap Q=\varnothing.$

Prove that for each $\sigma_0,\sigma_1\in \Sigma_k$ if $\sigma_1Q\neq\sigma_2Q$  then $\sigma_1Q$ and $\sigma_2Q$ are separated. It is sufficient to prove that for $\sigma_1=e, \sigma=\sigma_2\not\in\prod_{i=1}^m\Sigma_{\lambda_i}.$
Denote by $\pi_j\colon S^k\to S$ a canonical projection onto $j$-th component of product.

Suppose that there exist $x\in \overline{Q}\cap Q.$ It is equivalent to the assumption that 
for each $j\in\{1,\ldots,k\}$ $\pi_j(x)\in\overline{S_r}$ and $\pi_j(x)\in S_q,$ where $S_r$ and $S_q$ are $j$-th components of $Q$ and $\sigma Q$ respectively.

There are $3$ possible cases.
\begin{enumerate}
 \item $S_r$ is closed, then either $\overline{S_r}\cap S_q=S_r\cap S_q=\varnothing$ or $S_r=S_q,$ 
 \item $S_r$ and $S_q$ are open. Hence they are either separated or equal,
 \item $S_r$ is open, $S_q$ is closed. Then $\sigma$ maps coordinate corresponding to the closed $S_q$ to the component corresponding to the open $S_r.$ Hence there exist another component, with coordinate corresponding to the open $S_t$ mapped to the coordinate corresponding to the closed $S_h.$ So we are in the first case, while equality is not possible, as $S$ is connected. 
\end{enumerate}

Hence if $Q\neq\sigma Q$ then each point from $\overline{Q}$ differs from points of $Q$ on at least one coordinate. 
Therefore if $G$ is closed in $Q$ then it is closed in $\kappa_{k}^{-1}(Im\,\varphi).$ Therefore $\kappa_k|_{Q}$ is a continuous closed map.
Consider the map $(\prod_{i=1}^m\kappa_{\lambda_i})|_{Q}.$ It is a quotient map for a finite group action. Hence it is continuous and closed.

Note that 
$$\varphi\circ(\prod_{i=1}^m\kappa_{\lambda_i})|_{Q}=\kappa_{k}|_{Q}.$$

Hence $\varphi$ is continuous and closed. Hence $\varphi$ is a homeomorphism onto image.

Note that by Proposition \ref{semialg} $\kappa_{k}$ and $\kappa_{\lambda_i}$ are semialgebraic. Proposition 6.12 \cite{DK1981} ensures that $(\prod_{i=1}^m\kappa_{\lambda_i})$ is semialgebraic.  By Theorem 6.10 \cite{DK1981} $\kappa_{k}|_{Q}$ and $(\prod_{i=1}^m\kappa_{\lambda_i})|_{Q}$ are semialgebraic. Since $(\prod_{i=1}^m\kappa_{\lambda_i})|_{Q}$ and $\kappa_{k}$ are surjective and $\varphi$ is injective, using Proposition 7.9 \cite{DK1981} we obtain that $\varphi$ is semialgebraic.

\end{proof}

To cover wider class of possible stability theories, such as aperiodicity, we need to make a slight formal generalisation of Lemma \ref{dispower}. 

\begin{lemma}\label{dispowertweak}
Let $S$ be a connected semialgebraic space. 

Suppose that $S=\cup_{i=1}^m\cup_{j=1}^{t_i}S^j_i$ is a union of semialgebraic subspaces such that $S^j_i\cap S^q_r=\varnothing$. Suppose that for each $i\in\{1,\ldots,m\}$ $\cup_{j=1}^{t_i} S^j_i,$  is either open or closed in $S$ and for each $j\in\{1,\ldots t_i\}$ $S^j_i$ is open or closed in $\cup_{j=1}^{t_i} S^j_i.$
Take some $$\lambda=(\lambda^1_1,\ldots,\lambda_1^{t_1},\ldots,\lambda_m^{1},\ldots,\lambda_m^{t_m})\in\N^{\sum_{i=1}^mt_i},$$ such that $\sum_{j=1}^{t_i}\lambda^j_i=k_i.$

Define a map 
$$\begin{array}{l}
\varphi\colon \prod_{i=1}^{m}\prod_{j=1}^{t_i}{S^j_i}^{(\lambda^j_i)}\to S^{(\sum_{i=1}^mk_i)},
\\
\varphi\colon(\{s^i_{j,1},\ldots s^{i}_{j,\lambda^{j}_{i}}\})_{i\in\{1,\ldots,m\},j\in\{1,\ldots,t_i\}}\mapsto \cup_{i=1}^{m}\cup_{j=1}^{t_i}\{s^i_{j,1},\ldots s^i_{j,\lambda^j_i}\}.\end{array}$$

Then $\varphi$ is a semialgebraic homeomorphism onto image.
\end{lemma}
\begin{proof}
Consider a family of morphisms 
$$
\varphi_{i},\colon \prod_{j=1}^{t_i}(S^j_i)^{(\lambda^j_i)}\to(\cup_{j=1}^{t_i} S^j_i)^{(k_i)},\quad i=1,\ldots,m
$$
and a morphism
$$
\psi\colon \prod_{i=1}^m(\cup_{j=1}^{t_i} S^j_i)^{(k_i)}\to S^{(\sum_{i=1}^{m}k_i)}  
$$

$\varphi_1,\ldots,\varphi_m,\psi$ are semialgebraic homeomorphisms by Lemma \ref{dispower}. 

Hence, by  Proposition 6.12 \cite{DK1981} and Theorem 6.10 \cite{DK1981}, $\psi(\varphi_1,\ldots,\varphi_m)$ is a semialgebraic homeomorphism. 

Now it is enough to note that $\varphi=\psi(\varphi_1,\ldots,\varphi_m).$
\end{proof}
Following proposition is a direct consequence of Lemma \ref{dispowertweak}.
\begin{proposition}\label{unmixing}
Let $S$ be a non-trivial (monic) stability theory. Suppose that each stratum of $Str(S)$ is open in its closure.
 Let $(k,l,m)_{S}$ be a $D_n^S$-stratum. Then it is semialgebraically homeomorphic 
 to $$\Omega_s^{(k)}\times\Omega_{ss}^{(l)}\times\Omega_{un}^{(m)}.$$
\end{proposition}
One can also get a refined version of that proposition:
\begin{proposition}\label{unmixingconnected}
Let $S$ be a non-trivial (monic) stability theory. Suppose that each stratum of $Str(S)$ is open in its closure. 
%

Then $(k,l,m)_{S}$ could be decomposed into disjoint union of connected components of type 

$$R\cong\prod_{i\in h_R\subseteq {\Omega_s}^{con}}i^{\lambda_i}\prod_{i\in t_R\subseteq {\Omega_{ss}}}i^{\lambda_i}\prod_{i\in w_R\subseteq {\Omega_{un}}^{con}}i^{\lambda_i}.$$

Here $\sum_{i\in h_R}\lambda_i=l,$ $\sum_{i\in t_R}\lambda_i=k,$ $\sum_{i\in w_R}\lambda_i=m,$ and $R$ varies over all possible triples of partitions.
\end{proposition}
\begin{proof}
Suppose that strata of $Str(S)$ are either open in its closure in the ambient space. Then strata of $Str(S)^{con}$ also have this property. Hence Lemma \ref{dispowertweak} is applicable.   
\end{proof}
\begin{definition}
Let $T^n=(S^1)^n$ be an $n$-dimensional torus. Denote by $q$-skeleton of $n$-torus $T^n_q$ a union of all $q$-dimensional coordinate subtorii $$\bigcup_{I\subseteq \{1,\ldots,n\}, |I|=q}\{(s_1,\ldots, s_n)\in T^n| \forall i\in I\,s_i=1\}.$$ 

Examples of skeleta of tori could be seen on Figure 1.2.
\end{definition}

Now we are able to describe topology of $D$-strata.

\begin{lemma}\label{setgraph}
 Let $S$ be a connected semialgebraic subset of the real plane $\R^2=\CN.$ Then there exist a connected graph $G(S)$ homotopy equivalent to the  $S.$ 
 Moreover $S$ is homotopy equivalent to the bouquet of $b_1(S)$ circles, where $b_1$ is the first Betti number of $S.$ 
\end{lemma}
\begin{proof}
Note that by Corollary 9.3.7 \cite{BPR1998} we can assume $S$ to be bounded. Hence, by Triangulation Theorem \cite{Cos2000} $S$ is homotopy equivalent to some finite planar simplicial complex $K.$

Jordan-Brouwer separation theorem \cite[ Ch. 4, Sec. 7, Theorem 15]{Spa1966} assures that each $2$-dimensional component of complex has $1$ connected external border, and, possibly some internal borders. Let there be $t$ internal borders for all of the $K$

Now we able to proceed with the geometric construction.

Take $2$-dimensional components $A$ and $B$ connected by the $1$-dimensional path $C.$ We can blow $C$ and move all internal borders through the $C$ to $A$ or $B.$ As $K$ is connected we can transform $K$ into a finite simplicial complex with only $1$ connected $2$-dimensional region.
Any internal border is a union of simplices. Hence it could be seen as a set of empty borders of $2$-simplices, possibly with some $1$-complexes attached from the inside, 

Blowing these borders of $2$-simplices, moving them accordingly, we get a homotopy equivalence of $K$ to the bouquet of $t$ circles with some $1$-complexes attached. Transforming loops into cycles of minimal length and reducing all paths to a minimal possible length we've obtained a graph $G(S).$

Applying \cite[ Ex. 3.3.2.1]{Sti1993} we get the Lemma. 
\end{proof}
\begin{definition}[\cite{Ber1962}]
Recall that {\it cyclomatic number} $cycl(R)$ of the graph $R$ with $m$ edges, $n$ vertices and $c$ connected components is the number $m-n+c$ which is equal to the number of edges that does not belong to the spanning forest of $R.$
\end{definition}
\begin{theorem}[Homotopy type of a $D$-stratum]\label{hott}
Let $S$ be a non-trivial 
(mo\-nic) stability theory. Suppose that each stratum of $Str(S)$ is open in its closure.

Denote by $\vee$ an operation of taking bouquet of pointed topological spaces (gluing them at the marked point).

Let  

$$R\cong\prod_{i\in h_R\subseteq \Omega_s^{con}}i^{(\lambda_i)}\prod_{i\in t_R\subseteq \Omega_{ss}}^{con}i^{(\lambda_i)}\prod_{i\in w_R\subseteq \Omega_{un}^{con}}i^{(\lambda_i)}.$$

be a connected component of $(k,l,m)_{S}.$ Put $C=h_R\cup t_R\cup w_R.$

Here $\sum_{i\in h_R}\lambda_i=l,$ $\sum_{i\in t_R}\lambda_i=k,$ $\sum_{i\in w_R}\lambda_i=m,$ and $R$ varies over all possible triples of partitions.

Denote by $F_k$ a free group with $k$ generators.


Then $R$ is homotopy equivalent to

$$
(S^1)^{\sum_{i\in C, \lambda_i>1, 0<b_1(i)\leq \lambda_i}b_1(i)}\times\prod_{\begin{array}{l}i\in C\\ b_1(i)>\lambda_i>1\end{array}}T_{\lambda_i}^{b_1(i)}\times\prod_{\begin{array}{l}i\in C\\ \lambda_i=1\end{array}} \vee_{j=1}^{b_1(i)}S^1
$$

Fundamental group of $R$ is isomorphic to the 
$$\Z^{\sum_{i\in C, \lambda_i>1}b_1(i)}\times\prod_{i\in C,\lambda_i=1}F_{b_1(i)}.$$ 
\end{theorem}
\begin{proof}
Decomposition of $(k,l,m)_{S}$ into the connected components follows from Lem\-ma \ref{unmixingconnected}. Note that by Lemma \ref{setgraph} these connected components are homotopy equivalent to the bouquets of circles (which may consist from just one circle or from zero circles \-- contractible case). 

One can use Theorem 1.2 \cite{Ong2003} to determine homotopy type of each component. 

Fundamental group of $T_k^q$ is $\Z^q$ by Theorem 3 \cite{Hat1975}(or, equivalently one can use \cite{KT2013} and Hurewicz isomorphism theorem \cite[ Ch.7 Sec.5 Theorem 5]{Spa1966}) and the fundamental group of $\vee_{k=1}^qS^1$ is $F_q.$ 

Remember that the fundamental group of a product of spaces is a product of fundamental groups \cite[ Ch.2 Ex G.1.]{Spa1966}.

Finally, note that the fundamental group of $G(h)$ is free with cyclomatic number of generators \cite[ Ch.2 Sec.7 Corollary 5]{Spa1966}.
Applying Hurewicz Isomorphism Theorem \cite[ Ch.7 Sec.5 Theorem 5]{Spa1966} we get that for closed $h\in Str(S)^{con}$ $b_1(h)=cycl(G(h)).$ 
\end{proof}

\begin{theorem}[Betti numbers of a $D$-stratum]\label{betti}
Let $S$ be a non-trivial 
(monic) stability theory. Suppose that each stratum of $Str(S)$ is open in its closure. 
Denote as $b_i^j$ $i$-th Betti number of $j$-th stratum of $S.$
Then $u$-th Betti number of stratum $(k,l,m)_S$ is 

$$
\begin{array}{l}
\sum_{\begin{array}{l}r+t+q=u,\\0\leq r\leq k\\0\leq q\leq l\\ 0\leq t\leq m\end{array}}{b^s_1\choose r}{b^{ss}_1\choose q}{b^{un}_1\choose t}\times\\\qquad\qquad\qquad\times{b_0^{s}+k-r-1\choose k-r}{b_0^{ss}+l-q-1\choose l-q}{b_0^{un}+m-t-1\choose m-t}
\end{array}
$$
\end{theorem}
\begin{proof}
Proposition \ref{unmixing} guarantees that to compute Betti numbers of stratum $(k,l,m)_S$ it is enough to compute Betti numbers of symmetric products of $S$-stratum. 

In order to do that one can use I.G. Macdonald theorem \cite{McD1962a}, which gives a Poincar\'e polynomial of symmetric product and K\"unneth formula \cite{Kun1923}, which guarantees that Poincar\'e polynomial of product of spaces is product of Poincar\'e polynomials.

Namely, for each stratum of $S$ we can easily compute its Betti numbers, $b_0^j$\--- number of connected components, $b_1^j$\--- number of holes. As, according to Lemma \ref{setgraph} each stratum of $S$ is homotopy equivalent to the disjoint union of bouquets of circles (may be $0$ or $1$ circle), all higher Betti numbers are zeros.

McDonald theorem gives the generating function of a Poincar\'e polynomial:
$$
\begin{array}{l}
\frac{(1+xt)^{b_1}}{(1-t)^{b_0}}=(1+xt)^{b_1}(\sum_{w=0}^{\infty}{b_0^j+w-1\choose w}t^w)=\\=(\sum_{v=0}^{b_1^j}x^vt^v)(\sum_{w=0}^{\infty}{b_0^j+w-1\choose w}t^w)=\\=
\sum_{w=0}^{\infty}t^w(\sum_{v=0}^{w}x^v{b_1^j\choose r}{b_0^j+w-v-1\choose w-v}).
\end{array}
$$

Applying K\"unneth formula we obtain that the Poincar\'e polynomial of stratum is equal to

$$
\begin{array}{l}
(\sum_{r=0}^{k}x^r{b_1^s\choose r}{b_0^s+k-r-1\choose k-r})(\sum_{q=0}^{l}x^r{b_1^s\choose r}{b_0^s+l-q-1\choose l-q})\times\\\times(\sum_{t=0}^{m}x^r{b_1^s\choose t}{b_0^s+m-t-1\choose m-t})=\\=\sum_{u=0}^{k+l+m}x^u\sum_{r+t+q=u,0\leq r\leq k, 0\leq q\leq l, 0\leq t\leq m}({b^s_1\choose r}{b^{ss}_1\choose q}{b^{un}_1\choose t}\times\\\times{b_0^{s}+k-r-1\choose k-r}{b_0^{ss}+l-q-1\choose l-q}{b_0^{un}+m-t-1\choose m-t}).
\end{array}
$$
\end{proof}
Taking $u=0,$ we obtain the following proposition. 
\begin{proposition}\label{connectcomp}
 Let $S$ be a non-trivial 
 (monic) stability theory. Suppose that each stratum of $Str(S)$ is open in its closure. 
 
 Then stratum $(k,l,m)_S$ has ${b^{s}_0+k-1\choose k}{b^{ss}_0+l-1\choose l}{b^{un}_0+m-1\choose m}$ connected components.
\end{proposition}

Topology of strata for two most important cases: connected $\Omega_{ss}$ without self-intersections and pole placement case(finite $\Omega_{ss}$) could be described with higher precision. 

\begin{proposition}\label{compactconnected}
 Suppose that $\Omega_{ss}$ is homeomorphic to $S^1.$ E.g. in the case of compact convex closure of open $\Omega_s$, in the case  of Hurwitz stability theory, Schur stability theory or (quasi)hyperbolicity theory.
 
 Then strata $(k,0,m)$ are homeomorphic to $\R^{2(k+m)}$ and strata $(k,l,m), l>0$ are homeomorphic to $S^1\times D^{l-1}\times \R^{2(k+m)}$ if $l$ is odd and to $S^1\widetilde{\times}D^{l-1}\times\R^{2(k+m)}$ if $l$ is even.
 
 Here $\widetilde{\times}$ denote unique non-orientable bundle over $S_1$, $D^k$ is a closed disc.
\end{proposition}
\begin{proof}
This follows immediately from the Proposition \ref{unmixing} and Morton's theorem on symmetric product of a circle \cite{Mor1967}.
\end{proof}

For the description of geometry of pole placement problem one also need some definitions from the theory of subspace arrangements, which could be found in \cite{OT1992}.

\begin{proposition}\label{poleplacement}
Let $S$ be a stability theory with finite $\Omega_s,$ $|\Omega_s=r.|$
Consider a $D$-stratification of $S^{(n)}.$

\begin{enumerate}
 \item $\cup_{i=1}^{q}(i,n-i)_S$ is a general position arrangement of $r$ complex projective hyperplanes.
 \item $\bigcup_{q\leq s}(s,n-s)_S$ is 
 arrangement of ${r+q-1\choose q}$ $(n-s)$-dimensional complex projective subspaces; 
 
 Intersection poset of that arrangement is an intersection poset for the set of all multisubsets of $\{1,\ldots,r\}$ of cardinality no less then $q$ and no greater then $n.$
 Codimension of subspace intersection equals cardinality of corresponding multisubset.
 
\end{enumerate}
\end{proposition}
\begin{proof}
By Proposition \ref{proj}, union of strata  $\cup_{i=1}^{q}(i,n-i)_S$ is a set of all homogeneous binary forms having at least $1$ root at some points of $\Omega_s.$ Homogeneous binary form with roots at some specified points of $\CP^1$ defines a hyperplane in $\CP^{n},$ as that condition is linear on coefficients. As $k$-wise intersections are subspaces containing at least $k$-roots in some $k$ different points of $\Omega_s$ we get the first claim, which is the projectivised version of Theorem 1.1 \cite{Ong2003}.

Proof of claim 2, is analogous to the previous one. One should only note that the subspace of homogeneous binary forms having roots at each point of a fixed $q$-multisubset of $\Omega_s$ is a complex projective subspace of complex codimension $q$ and the number of distinct subspaces is equal to the number of $q$-multisubsets of $\Omega_s.$ As intersections between those subspaces are given by unions of multisubsets we obtain an intersection poset structure. 
\end{proof}

Now it is easy to describe topology of stratum belonging to the stratification $\underline{S^{(n)}}$ for some stability theory $S.$
\begin{proposition}\label{birthdeath1}
 Let $S$ be a stability theory.
 
 Then stratum $(k,l,m), k+l+m=r$ of $\underline{S^{(n)}}$ is semialgebraically homeomorphic to the stratum $(k,l,m)$ of monic stability theory $(S\setminus{\infty})^{(r)}.$
\end{proposition}
\begin{proof}
 By Proposition \ref{isolocal} stra\-tum $(k,l,m)$ is equal to the stratum with mul\-ti\-pli\-ci\-ty index $(k,l,m,n-r)$ of stra\-ti\-fi\-ca\-tion $(\{s\setminus \{e\}|s\in Str(S)\}\cup\{\{e\}\})^{(n)}.$
 
 As points are clo\-sed in $S,$ we can use Lem\-ma \ref{dispower}. Hen\-ce we ha\-ve $\underline{S^{(n)}}\cong\cup_{i=0}^nS^{(n-i)}$, where $S^{(n-i)}$ is a union of strata $(k,l,m,i)$ for all $k,l,m\in\N, k+l+m=r.$
\end{proof}

This is by no way a complete description of topology of $D^n_S$ strata. Using results by R.J. Milgram \cite{Mil1969} one can try to compute cohomology ring of stratum. Theorems of A. Hattori \cite{Hat1975} opens a possibility of computing higher homotopy groups of a stratum. 

Moreover, even more important set of open questions is the description of the structure of borders and higher-dimensional corners of the $D^n_S$-stratum. Careful study of refined stratifications $\widehat{S}$ and $\widehat{\underline{S}}$ is important here.
\section{Geometry of adjacency}
In order to study geometry of adjacencies between strata one can proceed from the topological structure to the graph-theoretic one. As that transfer is functorial it is possible to transfer symmetric product along the way. 

Examination of that construction and its consequences is the content of that section.
\begin{definition}\label{adjdig}
Let $T$ be a stratified real algebraic variety with finite stratification $L=\{L_1,\ldots,L_l\}.$ 

Define {\it ad\-ja\-ce\-ncy digraph} $Adj((T,L))$ of $(T,L)$ as a digraph with $V(G_L^T)=\{1,\ldots,l\}$ as set of vertices. 

$i,j\in V(Adj((T,L)))$ are con\-nec\-ted by an ed\-ge $(i,j)\in E(Adj((T,L)))$ if and only if $\overline{L_i}\cap L_j$ is non-empty in euclidean topology. 

In that case we will write that $L_i$ is adjacent to $L_j.$ 
\end{definition}
\begin{definition}
{\it Filtered digraph} $G\colon G_0\overset{\gamma_0}{\to} G_1\overset{\gamma_1}{\to}\ldots$ is a sequence of embeddings of digraphs. 

Let $G$ be a digraph with marked vertice $e.$

Sequence of morphisms  
$$G\overset{\gamma_1}{\to}G^2\to\ldots,\quad\gamma_i\colon (v_1,\ldots,v_i)\mapsto (v_1,\ldots v_i,e)$$ is a  filtered digraph $(G,e)^{\infty}$ \-- {\it infinite product} of $(G,e).$

{\it Infinite symmetric product} of digraph $G$ with marked point $e$ is a filtered digraph $G^{(\infty)}$ given as sequence of quotients defined by filtered action by permutations of filtered group $\Sigma^{\infty}=\Sigma_1\subset\Sigma_2\subset\ldots$ on infinite product of $G.$ 
\end{definition}

\begin{figure}[tbp]
\centering
\subcaptionbox{Schur and Hurwitz stabilities.}{$\xymatrix{\Omega_s\ar[r]&\Omega_{ss}&\Omega_{un}\ar[l]}$}\qquad
\subcaptionbox{Pole placement}{$\xymatrix{\Omega_s&\Omega_{un}\ar[l]}$}\\
\subcaptionbox{Aperiodicity}{$\xymatrix{\Omega_s\ar[rr]&&\Omega_{ss}\\&\Omega_{un}\ar[ur]\ar[ul]&}$}\\
\caption{Adjacency digraphs for different monic stability theories. Loops are not shown.}
\end{figure}
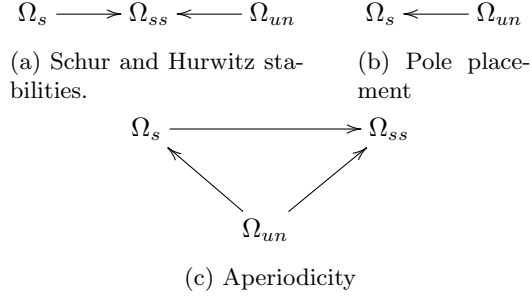 

\begin{lemma}\label{funct}
Adjacency graph is a functor from category of (filtered) stratified real algebraic varieties to the category of (filtered) digraphs.
\end{lemma}
\begin{proof}
Take a filtered stratified real algebraic variety $$(R,S)\colon (R_0,S_0)\overset{\lambda_0}{\to}(R_1,S_1)\overset{\lambda_1}{\to}(R_2,S_2)\to\ldots$$

Using definition of filtered stratified real algebraic variety we obtain an tautological embedding on sets of vertices of adjacency digraphs $Adj(\lambda_i)\colon S_i\to S_{i+1}.$

Moreover, as  for each $s\in S_i$ $\lambda_i(\overline{s})\subseteq\overline{\lambda_i(s)}\subseteq\overline{Adj(\lambda_i(s))},$ if $(s,t)$ is an edge in $Adj((R_i,S_i))$ then $(Adj(\lambda_i)(s),Adj(\lambda_i)(t))$ is an edge in $Adj((R_{i+1},S_{i+1})).$

Take a morphism of filtered stratified real algebraic varieties $\zeta\colon (R,S)\to (T,Q):$
$$
\xymatrix{
(R_0,S_0)\ar[r]^{\lambda_0}\ar[d]^{\zeta_0}&(R_1,S_1)\ar[r]^{\lambda_1}\ar[d]^{\zeta_1}&(R_2,S_2)\ar[d]^{\zeta_2}\ar[r]^{\lambda_2}&\ldots\\
(T_0,Q_0)\ar[r]^{\zeta_0}&(T_1,Q_1)\ar[r]^{\zeta_1}&(T_2,Q_2)\ar[r]^{\zeta_2}&\ldots
}
$$

Any $\zeta_i$ induces morphisms on set of vertices of adjacency digraphs. The same arguments as used for the morphisms defining filtrations shows that $Adj(\zeta_i)$ is a morphism of digraphs.

Finally, note that morphism of digraphs is completely defined by map on it's sets of vertices. This finishes the proof.
\end{proof}
\begin{figure}[H]
\centering
\subcaptionbox{Schur and Hurwitz stabilities.}{
$\xymatrix{\{2\Omega_s\}\ar[d]\ar[r]&\{2\Omega_{ss}\}&\{2\Omega_{un}\}\ar[d]\ar[l]\\
\{\Omega_s,\Omega_{ss}\}\ar[ur]&\{\Omega_{un},\Omega_{s}\}\ar[u]\ar[r]\ar[l]&\{\Omega_{ss},\Omega_{un}\}\ar[ul]}$}\\
\subcaptionbox{Pole placement}{$\xymatrix{\{2\Omega_s\}&&\{2\Omega_{un}\}\ar[dl]\ar[ll]\\&\{\Omega_s,\Omega_{un}\}\ar[ul]&}$}\\
\subcaptionbox{Aperiodicity}{$\xymatrix{&\{2\Omega_{un}\}\ar[ddr]\ar[ddd]\ar[ddl]\ar[dl]\ar[dr]&\\\{2\Omega_s\}\ar[d]\ar[rr]&&\{2\Omega_{ss}\}
\\\{\Omega_s,\Omega_{ss}\}\ar[urr]&&\{\Omega_{ss},\Omega_{un}\}\ar[u]\ar[ll]\\&\{\Omega_s,\Omega_{un}\}\ar[uul]\ar[uur]\ar[ul]\ar[ur]&}$}
\caption{Adjacency digraphs for the space of monic quadratic polynomials relative to the different monic stability theories. Loops are not shown.}
\end{figure} 

It should be noted that given definition of a symmetric product of the digraph differs from one considered for non-oriented graphs in \cite{AGRR2007}. In the latter paper authors consider only ``non-singular'' part of a symmetric product, that corresponds to vertices with non-repeating components.

\begin{proposition}\label{spoweradj}
Let $T$ be a stratified real algebraic variety with finite stratification $L$ and marked point $e\in F\in L.$

Take some $n\in\N\cup\{\infty\}$ such that $T^{(n)}$ is real algebraic variety for each $m\leq n$ .

Then there exist an isomorphism of filtered digraphs  $$\tau_n\colon Adj((T,L))^{(n)}\to Adj(T^{(n)},L^{(n)}).$$
\end{proposition}
\begin{proof}
Let us fix some $1<m\leq n.$ Prove that $Adj((T,L)^m)$ isomorphic to $(Adj((T,L)))^m.$ Namely, consider $Q=(Q_1,\ldots,Q_m), M=(M_1,\ldots M_m)\in L^m.$ Then $$\overline{\prod_{i=1}^mQ_i}\cap(\prod_{i=1}^mM_i)=(\prod_{i=1}^m\overline{Q_i}\cap M_i).$$

Consider symmetric product morphism $\kappa_i\colon (T,L)^m\to (T,L)^{(m)}.$ By Lem\-ma \ref{funct} there is a digraph morphism $Adj(\kappa_m)\colon Adj((T,L))^m\to Adj((T,L)^{(m)}).$

It is obvious that $V(Adj((T,L)^{(m)}))=V(Adj((T,L))^{(m)})$ and that $$E((Adj((T,L)))^{(m)})\subseteq E(Adj((T,L)^{(m)})).$$ $\kappa_m$ is quotient map for an action of finite group and hence closed. Therefore we get $E((Adj((T,L)))^{(m)})\subseteq E(Adj((T,L)^{(m)})).$

Finally Lemma \ref{funct} transforms our isomorphism into an isomorphism of filtered digraphs.
\end{proof}

Examples for adjacency digraphs of stability theories and for stratified spaces of polynomial coefficients could be seen on Fig.3 and 4, respectively.

One can also easily solve the problem of determining the presence of an edge between any two vertices of the symmetric product's adjacency digraph.
\begin{theorem}[Criterium of adjacency between strata]\label{adjcrit}
 Let $(R,S)$ be a stratified real algebraic variety with finite stratification $S=\{s_1,\ldots,s_k\}$ Suppose that for each $m\leq n$ $R^{(m)}$ is a real algebraic variety.
 Let $\tau=(t_1,\ldots,t_k)$ and $\eta=(q_1,\ldots, q_k),$ $\sum_{i=1}^kt_i=\sum_{i=1}^kq_i=m$ be strata of $(R,S)^{(m)}.$
 
 Then $(\tau,\eta)\in E(Adj((R,S))^{(m)})$ iff there exists such a family of natural numbers $\{\mu_t|t\in E(Adj(R,S))\}$ that the following system of linear equations has a solution:

$$
\begin{array}{l}
\sum_{j\colon (j,i)\in E(Adj(R,S))}\mu_{(j,i)}=e_i,\\
\sum_{j\colon (i,j)\in E(Adj(R,S))}\mu_{(i,j)}=t_i,\qquad i\in\{1,2,3\}.
\end{array}
$$
\end{theorem}
\begin{proof}
Note that, by Proposition \ref{spoweradj}, $(\tau,\eta)\in E(Adj((R,S)^{(m)}))$ iff there exist a sequence on $m$ pairs $w=((i_1,j_1),\ldots,(i_m,j_m)),\quad i_f,j_f\in \{1,\ldots,k\}$ such that for each $r\in\{1\ldots m\}\quad$ $(s_{i_r},s_{j_r})\in E(Adj(S))$ and $|\{r|i_r=x\}|=t_x,\quad |\{r|j_r=v\}|=q_v.$

Take some $w.$ Note that, if one try to compare $w$ with a sequence of loops, appearance of each pair $(i,j)$ in $w$ decrease multiplicity of vertice $s_i$ in $\tau$ by $1$ and increase multiplicity of $s_j$ by one. Proceeding, by changing all of the loops into an edges of $w,$ one will get a decrease of multiplicity of each vertice by the number of outgoing edges in a sequence and increase of that by the number of ingoing edges.

So the set of edge multiplicities of the sequence $w$ is a solution of the linear system in consideration.
\end{proof}

Thus we've described geometry of adjacency for $D$-stratifications:
\begin{theorem}[Geometry of adjacency graph]\label{adjdigstab}
 Let $S$ be a (monic) stability theory. Then $Adj(S^{(n)})\cong Adj(S)^{(n)},$ $Adj(\underline{S^{(n)}})\cong Adj((S,\underline{Str(S)}))^{(n)}.$
\end{theorem}
\begin{proof}
 First isomorphism is follows immediately from the Proposition \ref{spoweradj}, while the second is a consequence of Proposition \ref{isolocal} and Proposition \ref{spoweradj}.
\end{proof}

One can also describe all possible adjacency graphs for stability theories.
\begin{proposition}\label{AdjS}
 Let $G$ be a mar\-ked di\-graph with at most three ver\-ti\-ces mar\-ked by some subset of a set $\{s, ss, un\}.$

 Suppose that following conditions hold:
 \begin{enumerate}
  \item There is a loop in each vertex.
  \item $G$ is weakly connected.
  \item There are no ingoing edges at vertex $un.$
  \item If $\{s, ss\}\subset V(G)$ then there is an edge $(s,ss).$
  \item If $ss\in V(G)$ then $s\in V(G).$
 \end{enumerate}
  Then there exist a stability theory $S$ with $Adj(S)\cong G$ with isomorphism sending appropriately marked vertices into corresponding vertices of $Adj(S).$
  Moreover, each adjacency graph of stability theory satisfies these conditions.
\end{proposition}
\begin{proof}
If $G$ have only one vertice, then one can take $\Omega=\varnothing$ or $\Omega=\CN\bf{ P}^1$ depending on marking.
Suppose that $G$ has two vertices. If these vertices has markings $s, un$ the it is enough to take quasihyperbolicity or Hurwitz quasistability.
If markings are $s,ss,$ then one can take $\Omega_{ss}=[0,1]$ if there are no edge $(ss, s)$ and $\Omega_{ss}=[0,1)$  if there is one.

Suppose that  $G$ has $3$ vertices. If there are both edges $(ss, s)$ and $(un,s)$ one can take $\Omega=\{z|Im\, z<0\}\cup\{\infty\}\cup\{iz|z\in\R, z\leq 0\}.$
If there are no edge $(ss, s)$ but there exist edge $(un,s)$ it is possible to take $\Omega=\{z|Im z<0\}\cup\{iz|z\in\R, z<0\}.$
If there are no edges of both types one can take Hurwitz stability.

Now we can prove that proposition in the other direction. Since if $T$ in non-empty then $\overline{T}\cap T$ is also non-empty, the first condition holds. $\CP^1$ is connected. Hence second condition holds. Third condition follows from the fact that $\Omega_{un}$ is open, while the fourth and fifth from an embedding $\Omega_{ss}\subset\overline{\Omega_s}.$  
\end{proof}
Proof of the next proposition is completely analogous to the previous one.
\begin{proposition}\label{Adjunderline}
Let $G$ be a marked digraph with at most four vertices marked by some subset of a set $\{s, ss, un, \infty\}.$
 Then there exist a stability theory $S$ with $Adj(\underline{S})\cong G$ with isomorphism sending appropriately marked vertices into corresponding vertices of $Adj(\underline{S})$ iff following conditions hold:
 \begin{enumerate}
  \item There is a loop in each vertex, 
  \item $G$ is weakly connected,
  \item There are no ingoing edges at vertex $un,$
  \item If $\{s, ss\}\subset V(G)$ then there is an edge $(s,ss).$
  \item If $ss\in V(G)$ then $s\in V(G).$
  \item $\infty\in V(G)$
  \item There are no outgoing edges from $\infty.$
 \end{enumerate}
\end{proposition}
\section{Standard stability theories}
Matrix-polynomial duality and corresponding duality between monic and non-monic stability theories allows us to axiomatize classical stability theories.

Invariance under complex conjugation means that the stability theory in consideration has invariant structure for the case of polynomials with real coefficients, namely that conjugate pairs of roots always belong to the same stratum.

\begin{theorem}[Standard stability theories]\label{standard}
 Let $S$ be a stability theory with non-empty $\Omega_{s}$ and $\Omega_{ss}=\partial\Omega_s=\partial\Omega_{un}.$
 
Let, moreover following conditions holds:
 \begin{enumerate}
  \item  $\Omega_{ss}$ is a non-empty irreducible real algebraic curve without isolated points.
  \item Inversion $\lambda\mapsto\frac{1}{\lambda}$ is an automorphism of stratified space $S.$
  \item  Complex conjugation $\lambda\mapsto \overline{\lambda}$  is an automorphism of stratified space $S.$
  \item  $0$ and $\infty$ aren't both stable or both unstable.
 \end{enumerate}
 Then, up to the interchange between $\Omega_s$ and $\Omega_{un}$ or getting their union, $S$ is either Hurwitz stability theory( in the case of union we obtain Fenichel stability), Schur stability theory or hyperbolicity theory(with $\Omega_{ss}$ as a real line).
 These three will be called {\it standard stability theories.}
\end{theorem}
\begin{proof}
Note that instead of invariance under inversion $\lambda\mapsto\frac{1}{\lambda}$ and conjugation $\lambda\mapsto\overline{\lambda}$ one can use invariance under conjugation and conjugate of inversion
$\lambda\mapsto\frac{1}{\overline{\lambda}}.$
Take some defining polynomial $f(x,y)$ of curve $\Omega_{ss}$ and write it in the polar coordinates $x\mapsto r\cos\varphi,\quad y\mapsto r\sin\varphi.$
Invariance properties lead us to the following equivalences:
$$\begin{array}{l}f(r,\cos\varphi,\sin\varphi)=0\Leftrightarrow f(\frac{1}{r},\cos\varphi,\sin\varphi)=0,\\f(r,\cos\varphi,\sin\varphi)=0\Leftrightarrow(r,\cos\varphi,-\sin\varphi)=0\end{array}$$

By Theorem 2 \cite{MR2008} for each $\varphi\in [0,2\pi)$ $f$ can be decomposed into product $f=hq,$ such that $h$ consists from all only non-negative real roots of $f$ and either antipalindromic or palindromic on $r.$ Invariance under inversion implies that either $\{0,\infty\}\subset\Omega_{ss}$ or $\{0,\infty\}\cap\Omega_{ss}=\varnothing.$

Suppose that $f$ has different signs on $0$ and on $\infty$ or both $0$ and $\infty$ belongs to the curve defined by $f.$
In the latter case $h$ in polar coordinates is monomial on $r.$ Hence $h$ does not depend on $r.$ Hence $h$ either defines an isolated point $(0,0)$ or $h$ defines a ray. In the first case one can take another $\varphi$ and get the same situations, or the situation when $h$ defines a ray. If $h$ defines a ray then, assuming that $\Omega_{ss}$ is irreducible curve we get $f$ is either $x$ (quasihyperbolicity), or $y$ (Hurwitz stability).

Assume now that $f$ has different signs of $0$ and on $\infty.$ Note that $f$ changes sign between $0$ and $\infty$ iff $h$ changes sign. Hence $h$ is antipalindromic on $r.$
Hence, by Lemma 3 \cite{MR2008}  there exist a palindromic polynomial $g(r,\varphi)$ such that $h=(r-1)g.$ Therefore the set of points with $r=1$ is a subset of the zero set of $h.$ Note that, as  $\Omega_{ss}$ does not have isolated points and dependence between roots and coefficients is smooth, there are infinitely many such directions. Hence, as $\Omega_{ss}$ is irreducible, $f$ has a zero set defined by the polynomial $x^2+y^2-1.$

Now we need to prove that if $\Omega_{ss}$ is an irreducible real algebraic curve and $0, \infty$ belong to different strata then one can take such $f$ that $f(0)$ and $f(\infty)$ have different signs.

As $\Omega_{ss}$ is irreducible there exist an irreducible polynomial  $f(x,y)$ defining $\Omega_{ss}.$ 

Note that if $z=f(x,y)$ does not change sign on the direction  of transversal to irreducible curve $\Omega_{ss}$ at some non-singular point of $\Omega_{ss}$ then it does not change sign at all non-singular
points. Namely,  if it does not change sign at one non-singular point, then it does not change sign on infinitely many of them. Take derivatives until first non-zero one appear. Order of that derivatives will be
different on non-singular parts of a curve where $z=f(x,y)$ does not change size and where it does change the sign. That contradicts irreducibility.

Suppose that the function $z=f(x,y)$ does not change sign on a general position point of some connected component of $\Omega_{ss}.$ 
Consider an algebraic curve $C$ defined by equations $g_x=\frac{\partial f(x,y)}{\partial x},\quad g_y=\frac{\partial f(x,y)}{\partial y}.$
$C$ intersects $\Omega_{ss}$ at the infinite number of points. So it is equal to $\Omega_{ss}.$
Therefore $\Omega_{ss}$ is defined by some irreducible factor $f_1$ of $gcd(g_x, g_y).$ But $\deg f_1< \deg f.$  Proceed by induction.

So $\Omega_{ss}$ decomposes $\CP^1$ into two finite families of regions defined by the sign of the defining polynomial. 
That decomposition has the property that the border between regions that belong to the same family consists of finite number of points.

It is easy to prove that if there is a decomposition of plane with  given border and such a property exists then that decomposition is unique.
Namely, if one take a region and assign it to the family then, using induction, one can uniquely determine the family of any other region.

Recall that $\Omega_{ss}=\partial\Omega_s=\partial\Omega_{un}.$ Hence one of that families could be identified with $\Omega_s$ and the other with $\Omega_{un}.$
Hence if $0,$ $\infty$ belongs to different strata there exist $f(x,y)$ with different signs at $0$ and at $\infty.$
\end{proof}
Note that among $4$ conditions of theorem \ref{standard} first and the fourth are those that put some boundaries on class of stability theories.

 \afterpage{\clearpage}
\begin{figure}[!p]
\begin{tabular}{cc}
 {\animategraphics[loop,controls, scale=0.2]{6}{./figures/Animation/Animation1/1-frame-}{0}{99}}\\\\
 (a) \quad $\begin{array}{l}\Omega=\{(x^2+y^2)^2-cx(x^2+y^2)+\\+cx^2 +10c(x^2+y^2)-cx+1< 0\},\,c\in[-5,5].\end{array}$\\\\
 {\animategraphics[loop,controls, scale=0.2]{6}{./figures/Animation/Animation2/2-frame-}{0}{99}}\\\\
 (b) $\begin{array}{l}\Omega=\{(x^2+y^2)^2-cx(x^2+y^2)-x^2+\\ +\frac{3}{2}c(x^2+y^2)-cx+1< 0\},\,c\in[-5,5].\end{array}$\\\\
 {\animategraphics[loop,controls, scale=0.2]{18}{./figures/Animation/Animation3/3-frame-}{0}{549}}\\\\ 
 (c)\quad $\begin{array}{l}\Omega=\{(x^2+y^2)^3-cx(x^2+y^2)^2+cx^2(x^2+y^2)-\\-cx^3+cx(x^2+y^2)+cx^2-cx+1<0\},\,c\in[-5,50].\end{array}$
\end{tabular}
\phantomcaption
\end{figure}
\begin{figure}[!p]
\ContinuedFloat
\begin{tabular}{cc}
 {\animategraphics[loop,controls, scale=0.2]{9}{./figures/Animation/Animation4/4-frame-}{0}{199}}\\\\
 (d)\quad $\begin{array}{l}\Omega=\{(x^2+y^2)^3-cx(x^2+y^2)^2+\\+cx^2(x^2+y^2)-4c(x^2+y^2)^2-cx^3+cx(x^2+y^2)+\\+cx^2-4c(x^2+y^2)-cx+1 < 0\},\,c\in[-50,50].\end{array}$ \\\\
 {\animategraphics[loop,controls, scale=0.2]{9}{./figures/Animation/Animation7/7-frame-}{0}{199}}\\\\
 (e)\quad $\Omega=\{(x^2+y^2)-cx+1< 0\},\,c\in[-10,10].$\\\\
 {\animategraphics[loop,controls, scale=0.2]{12}{./figures/Animation/Animation6/6-frame-}{0}{299}}\\\\
 (f)\quad$\begin{array}{l}\Omega=\{(x^2+y^2)^3-cx(x^2+y^2)^2\\+cx^2(x^2+y^2)-2c(x^2+y^2)^2-cx^3\\+cx(x^2+y^2)+cx^2-2c(x^2+y^2)-cx+1 < 0\},\,c\in[-150,150].\end{array}$\end{tabular}
\caption{$1$-parametric families of conjugation and inversion invariant irreducible real curves.}
\label{families}
\end{figure}

Condition of irreducibility is of technical nature: any curve with $k$ irreducible components could be translated among the actions inversion and conjugation and produce some invariant border with at most $4k$ irreducible components. Moreover, any union of invariant curves is invariant. Namely, following holds:

\begin{proposition}\label{orbit}
 Let $S$ be a stability theory with  $\Omega_{ss}=\partial\Omega_s=\partial\Omega_{un}.$
 Let, moreover, following conditions holds:
 \begin{enumerate}
  \item $\Omega_{ss}$ is a real algebraic curve,
  \item Inversion $\lambda\mapsto\frac{1}{\lambda}$ is an automorphism of stratified space $S;$
  \item Complex conjugation $\lambda\mapsto \overline{\lambda}$  is an automorphism of stratified space $S.$
  \end{enumerate}
 
 Then there exist a polynomial $f(x,y)$ such that zero set of a polynomial
 
 $$
 F(x,y)=(x^2+y^2)^{\tau}f(x,y)f(x,-y)f(\frac{x}{x^2+y^2},\frac{y}{x^2+y^2})f(\frac{x}{x^2+y^2},-\frac{y}{x^2+y^2}),
 $$
 
 is an affine part of $\Omega_{ss}.$
 
 Here $\tau$ is a minimal integer from $[0,2 deg (f(x,y))]$ such that $F(x,y)$ is a polynomial.
 
 Moreover, zero set of each polynomial representable as $F(x,y)$ for some $f(x,y)$ satisfies conditions 1-3.
 \end{proposition}
 \begin{proof}

Note that the conditions 1-2 are the conditions of invariance of $\Omega_{ss}$ under action of $\faktor{\mathbb Z}{2\mathbb Z}\oplus\faktor{\mathbb Z}{2\mathbb Z}$ on $\CP^1.$ Set of points $$\{(x,y), (x,-y), (\frac{x}{x^2+y^2},\frac{y}{x^2+y^2}),(\frac{x}{x^2+y^2},-\frac{y}{x^2+y^2})\}$$ constitutes an orbit.

Hence, if zero set of $f(x,y)$ is $\Omega_{ss}$ then $F(x,y)$ has the same zero set.

Take some polynomial representable as $F(x,y).$ It's zero set is invariant under inversion and conjugation, as inversion and conjugation induces transposition of it's factors.
 \end{proof}
 
Fourth condition is more interesting. It assumes that the definition of stability has something to do with measuring a root, supposing that it is ``big'' or ``small'' in some sense.
What will happen if that condition will be dropped?

 \begin{proposition}\label{palind}
 Let $S$ be a stability theory with $\Omega_{ss}=\partial\Omega_s=\partial\Omega_{un}.$
 Let, moreover, following conditions holds:
 \begin{enumerate}
  \item $\Omega_{ss}$ is real algebraic curve.
  \item Inversion $\lambda\mapsto\frac{1}{\lambda}$ is an automorphism of stratified space $S.$
  \item Complex conjugation $\lambda\mapsto \overline{\lambda}$  is an automorphism of stratified space $S.$
  \item $0$ and $\infty$ are both either stable or unstable.
  \item There exists a polynomial $f(x,y)$ with zero set $\Omega_{ss}$ such that for each $k\in{\mathbb R}\cup\{\infty\}$ {\it complex} roots of $f(x,kx)$ are inversion-invariant and $f(x,y)$ is even on $y.$
  \end{enumerate}
Then $\Omega_{ss}$ could be represented as zero set of an even degree polynomial: 
$$
\sum_{i=0}^{\frac{n}{2}}\sum_{j=0}^{\lfloor\frac{i}{2}\rfloor}a_{ij}x^{i-2j}(x^2+y^2)^{j}(1+(x^2+y^2)^{\frac{n}{2}-i})
$$

Some examples of families polynomials of that type could be seen on Figure \ref{families}.
\end{proposition}
\begin{proof}

Write $f(x,y)$ in polar coordinates. Using Condition 5 and Theorem 2 \cite{MR2008}, recalling that all antipalindromic possibilities have been explored in proof of Theorem 1, we obtain that $f(r,\cos\varphi,\sin\varphi)$ is palindromial as polynomial on $r$ and is also a polynomial on $r\cos\varphi,r\sin\varphi.$

Using the fact that $$\R[r,\cos\varphi,\sin\varphi]\cong \faktor{\R[r,t,q]}{\langle t^2+q^2-1\rangle}$$ and that $f(x,y)$ is ev\-en on $y$ we obtain
$$
f(x,y)=\sum_{i=0}^{\frac{n}{2}}\sum_{j=0}^{\lfloor\frac{i}{2}\rfloor}a_{ij}x^{i-2j}(x^2+y^2)^{j}(1+(x^2+y^2)^{\frac{n}{2}-i})
$$

\end{proof}

That section is {\it de facto} devoted to the study of irreducible real algebraic curves on $\CP^1$  invariant under the action $\faktor{\mathbb Z}{2\mathbb Z}\oplus\faktor{\mathbb Z}{2\mathbb Z}$ via inversion and conjugation. 

That kind of questions could be understood as a special case of a seemingly open problem.
\begin{problem}
Let $G$ be a finite subgroup of a M\"obius group of fractional-linear transformations acting on $\CP^1.$

How to describe $G$-invariant irreducible real algebraic curves on $\CP^1$? 

Note that these groups could be completely classified \cite{Tot2002}.
\end{problem}

\end{document}